\documentclass[reqno]{amsart}
\usepackage{graphicx} 
\numberwithin{equation}{section}
\usepackage{amsmath}
\usepackage{amssymb}
\usepackage{amsthm}
\usepackage{color}
\usepackage{enumerate}
\theoremstyle{definition}
\newtheorem{thm}{Theorem}[section]
\newtheorem{prop}[thm]{Proposition}
\newtheorem{rem}[thm]{Remark}
\newtheorem{lemma}[thm]{Lemma}
\newtheorem{definition}[thm]{Definition}

\newtheorem{claim}{Claim}
\usepackage{mathrsfs}

\newcommand{\map}{f_a}

\newcommand{\Leb}{\text{Leb}}
\newcommand{\llya}{\underline{\lambda}}
\newcommand{\ulya}{\overline{\lambda}}
\newcommand{\lya}{\lambda}
\newcommand{\kneading}{\underline{e}}

\title[Existence of the Lyapunov exponent]{Existence of the Lyapunov exponent for $S$-unimodal maps}
\author{Yuya Arima}
\thanks{\it Affiliation: Graduate School of Mathematics, Nagoya University,
Furocho, Chikusa-ku, Nagoya, 464-8602, JAPAN} 
\thanks{E-mail address: yuya.arima.c0@math.nagoya-u.ac.jp}
\subjclass[2020]{37A05, 37D25, 37E05}
\thanks{{\it Keywords}: Lyapunov exponents, unimodal maps, Birkhoff averages}

\begin{document}

\begin{abstract}
In this paper, we show that for any $S$-unimodal map $T$ on $[0,1]$  with a non-flat critical point, the Lyapunov exponent exists for Lebesgue almost every point and is equal to a constant $\lambda_T$. 
Moreover, $\lambda_T=0$ if and only if $T$ admits neither an absolutely continuous $T$-invariant probability measure with positive entropy nor a strictly stable periodic orbit.
Our result holds even in the absence of statistical properties.
This is in sharp contrast with the two-dimensional setting, where, in general, Lyapunov exponents need not exist almost everywhere. 
In particular, if an $S$-unimodal map with a non-flat critical point is infinitely renormalizable or non-statistical, then for Lebesgue almost every $x\in [0,1]$ the Lyapunov exponent along the orbit of $x$ exists and is equal to $0$. 
A key ingredient is the following result of independent interest. If an $S$-unimodal map with a non-flat critical point has no periodic attractor then  for Lebesgue almost every $x\in [0,1]$ the lower Lyapunov exponent along the orbit of $x$ is non-negative. In particular, in the absence of periodic attractors, exponential contraction cannot occur along the orbit of Lebesgue almost every point. 
\end{abstract}

\maketitle

\section{Introduction}
Lyapunov exponents are fundamental quantities that characterize chaotic behavior in dynamical systems. However, Lyapunov exponents are defined through asymptotic growth rates of derivatives along orbits. Therefore, it is natural to ask whether the set of initial points for which the Lyapunov exponent exists is physically observable. Motivated by this question, several authors have investigated this problem. See, for example, \cite{KirikiLiNakanosoma, kiriki2026non, YushiSomaYamamoto, OttYorke}.
In particular, in the two-dimensional setting, Kiriki, Li, Nakano, and Soma \cite{KirikiLiNakanosoma} showed that there is an abundance of dynamical systems for which the set of initial points for which the Lyapunov exponent does not exist has positive Lebesgue measure. 
In contrast to the two-dimensional setting, we show that for any $S$-unimodal map $T$ on $[0,1]$  with a non-flat critical point, the Lyapunov exponent exists for Lebesgue almost every point and is equal to a constant $\lambda_T$.
This highlights a striking difference between one-dimensional and two-dimensional dynamical systems concerning the existence of Lyapunov exponents.
Previous results guaranteeing the existence of the Lyapunov exponent for Lebesgue almost every point typically rely on strong statistical or dynamical assumptions, such as the existence of an absolutely continuous invariant probability measure or a periodic attractor.
We remove all such assumptions.
In particular, our result holds even in the absence of statistical properties. 
The proof relies on a new lower Lyapunov estimate of independent interest (Proposition \ref{prop nonegative}).

We say that a differentiable map $T:[0,1]\rightarrow [0,1]$ is a unimodal map if there exists a unique point $c\in (0,1)$ such that $T'(c)=0$. Moreover, a unimodal map $T:[0,1]\rightarrow [0,1]$ is said to be of class $C^3$ if $T$ is $C^1$ on $[0,1]$ and $C^3$ on $[0,1]\setminus\{c\}$.
Let $T$ be a unimodal map such that $T$ is of class $C^3$. 
For $x\in (0,1)\setminus\{c\}$ we define 
\[
ST(x):=\frac{T'''(x)}{T'(x)}-\frac{3}{2}\left(\frac{T''(x)}{T'(x)}\right)^2.
\]
We say that $T$ has negative Schwarzian derivative if for all $x\in(0,1)\setminus\{c\}$ we have $ST(x)< 0$.

A unimodal map $T:[0,1]\rightarrow [0,1]$ is called an $S$-unimodal map if it is of class $C^3$, has negative 
Schwarzian derivative, and satisfies $T(0)=T(1)=0$, $T'(0)>1$ and $T'(1)\neq 0$. 
An $S$-unimodal map $T$ is said to have a non-flat critical point if 
there exists 
a $C^3$ diffeomorphism $\phi:\mathbb{R}\rightarrow \mathbb{R}$ with $\phi(0)=0$  and a real number
$l\in(1,\infty)$  such that for $x$ close to $c$ we have 
\[
T(x)=T(c)- |\phi(x-c)|^{l}.
\]
The value of $l$ is known as the critical order of $c$.  
We denote by $\Leb$ the Lebesgue measure on $[0,1]$. For each $x\in [0,1]$ we define the lower Lyapunov exponent $\llya(x)$ by 
\[
\llya(x):=\liminf_{n\to \infty}\frac{1}{n} \log |(T^n)'(x)|
\]
and the upper Lyapunov exponent $\ulya(x)$ by 
\[
\ulya(x):=\limsup_{n\to \infty}\frac{1}{n} \log |(T^n)'(x)|.
\]
If these two limits coincide, we define the Lyapunov exponent $\lya(x)$ by 
\[
\lya(x):=\llya(x)=\ulya(x).
\]
We say that $T$ has a strictly stable periodic orbit if there exists a periodic point $p\in[0,1]$ with period $q\in\mathbb{N}$ such that $|(T^q)'(p)|<1$.
The map $T$ has an absolutely continuous $T$-invariant probability measure if there exists a $T$-invariant Borel probability measure $\mu$ such that $\mu$ is absolutely continuous with respect to $\Leb$.

\begin{thm}\label{thm main}
    For every $S$-unimodal map $T$ with a non-flat critical point there exists a constant $\lya_T\in \mathbb{R}\cup\{-\infty\}$ such that for $\Leb$-almost every $x\in [0,1]$ we have 
    \begin{align*}
    \lya(x)=\lya_T.    
    \end{align*}
    Moreover, $T$ has an absolutely continuous $T$-invariant probability measure with  positive entropy if and only if $\lya_T>0$. Furthermore, $T$ has a strictly stable periodic orbit if and only if $\lya_T<0$.
\end{thm}

When an $S$-unimodal map $T$ with a non-flat critical point admits an absolutely continuous $T$-invariant probability measure or a periodic attractor, the existence of the Lyapunov exponent $\lambda(x)$ for $\Leb$-almost every $x\in [0,1]$ is known (see Section \ref{seq Technical lemmas for distortion estimates of $S$-unimodal maps} for the definition of periodic attractors). However, to the best of our knowledge, the existence of the Lyapunov exponent $\lambda(x)$ for $\Leb$-almost every $x\in [0,1]$ has not been established in all other cases. 
Therefore, the novelty of Theorem \ref{thm main} is that, even when an $S$-unimodal map $T$ with a non-flat critical point 
admits neither an absolutely continuous $T$-invariant probability measure nor a periodic attractor, for $\Leb$-almost every $x\in [0,1]$ the Lyapunov exponent $\lambda(x)$ nevertheless exists. In the absence of an absolutely continuous $T$-invariant probability measure or a periodic attractor there is in general no invariant probability measure describing the statistical behavior of typical orbits with respect to $\Leb$. Consequently, standard ergodic arguments do not apply.

In the proof of \cite[Proposition 1]{BruinRiveraShenStrien}, it was shown that if $T$ is a $C^3$ unimodal map with a non-flat critical point and all periodic points are hyperbolic  and repelling then $T$ is real-analytically conjugate to a map with negative Schwarzian derivative (see also the comment after \cite[Theorem D]{Alejo}). Thus, one can drop the assumption of negative Schwarzian derivative from Theorem \ref{thm main}, at the cost of adding the assumption that all periodic points are hyperbolic and repelling. 

In the following, we explain that, among $S$-unimodal maps with a non-flat critical point admitting neither an absolutely continuous $T$-invariant probability measure nor a periodic attractor, there are cases in which one naturally expects 
\begin{align}\label{eq intro}
\text{
$\lambda(x)=0$ for $\Leb$-almost every $x\in[0,1]$,
}
\end{align}
as well as cases in which such a conclusion is difficult to predict.
For the sake of simplicity, we consider quadratic maps. 
For $a\in [0,4]$ we define the quadratic map $\map:[0,1]\rightarrow [0,1]$ by 
\[
\map(x)=ax(1-x).
\]
It is easy to verify that for all $a\in(1,4]$ the map $\map$ is an $S$-unimodal map and the order of the critical point $c=1/2$ of $\map$ is $2$. For all $a\in [0,1]$ and $x\in [0,1]$ we have $\lim_{n\to\infty}\map^n(x)=0$.

Let $M([0,1])$ be the set of all Borel probability measures. We endow $M([0,1])$ with the weak* topology. 
For $x\in [0,1]$ we denote by $\delta_x$ the Dirac measure at $x$. 
For $a\in[0,4]$ an $\map$-invariant Borel probability measure $\mu$ on $[0,1]$ is said to be a physical measure if there exists a Borel set $A\subset [0,1]$ such that $\Leb(A)>0$ and for all $x\in A$ the sequence of measures $\left\{\frac{1}{n}\sum_{k=0}^{n-1}\delta_{\map^k(x)}\right\}_{n\in\mathbb{N}}$ converges to $\mu$.

We first consider the case where one naturally expects the conclusion \eqref{eq intro} to hold. For $a\in[0,4]$ we say that $\map$ is infinitely renormalizable if there exist a strictly increasing sequence $\{p_k\}_{k\in\mathbb{N}}\subset \mathbb{N}$ and a sequence $\{J_{k}\}_{k\in\mathbb{N}}$ of closed subintervals of $(0,1)$ with $J_{k+1}\subset J_{k}$ for all $k\in\mathbb{N}$ such that for each $k\in\mathbb{N}$ we have the following:   
\begin{itemize}
    \item[(R1)] The interiors of $J_k, \map(J_k),\cdots,\map^{p_k-1}(J_k)$ are disjoint.
    \item[(R2)] We have $\map^{p_k}(J_k)\subset J_k$ and $\map^{p_k}(\partial J_k)\subset \partial J_k$.
    \item[(R3)] There exists $i\in \{0,\cdots,p_k-1\}$ such that $c\in\text{Int}(\map^i(J_k))$. 
    \item[(R4)] If $J_k\subset J'\subset (0,1)$ is a closed interval satisfying the conditions (R1) to (R3) then $J'=J_k$. 
\end{itemize}
Here, $\partial A$ denotes the Euclidean boundary of $A\subset [0,1]$ and $\text{Int}(A)$ denotes the Euclidean interior of $A\subset [0,1]$. Let $\map$ be infinitely renormalizable. Then, by \cite[Chapter V, Theorem 1.6]{de2012one}, there exists a unique $\map$-invariant Borel probability measure $\mu_a$ supported on the solenoidal attractor. Moreover, the measure $\mu_a$ is a physical measure and has zero entropy. Furthermore, by Keller \cite[Theorem 7]{Kellerliting}, we have $\int \log |\map'| d\mu_a=0$ (see also \cite[Theorem B]{FariaGuarino}). Therefore, in this case we can strongly expect the conclusion \eqref{eq intro} to hold. However, because of the presence of the critical point, $\log|\map'|$ is neither continuous nor bounded. Therefore, the conclusion does not follow immediately and requires a proof. Theorem \ref{thm main} provides such a proof. 

Next, we consider the case where one cannot naturally expect the conclusion \eqref{eq intro} to hold.
For a Borel probability measure $\nu$ on $[0,1]$ let $\overline{\omega}_a(\nu)$ be the set of all weak accumulation points of the sequence 
\[
\left\{\frac{1}{n}\sum_{k=0}^{n-1}\nu\circ \map^{-k}\right\}_{n\in\mathbb{N}}.
\]
Let $d$ be a metric on $M([0,1])$ generating the weak* topology.
\begin{thm}\label{thm quadratic map}
    There exists 
    an uncountable set $\mathcal{S}\subset [0,4]$ such that for any $a\in \mathcal{S}$ we have the following:
    \begin{itemize}
        \item[(1)] For $\Leb$-almost every $x\in [0,1]$ we have         $\overline{\omega}_a(\delta_x)
=\overline{\omega}_a(\delta_{\map(c)})=\overline{\omega}_a(\Leb)   
       $.
       \item[(2)] $\overline{\omega}_a(\delta_{\map(c)})$ is not a singleton.
\item[(3)] There exist $\{n_k\}_{k\in\mathbb{N}}\subset \mathbb{N}$ with $\lim_{k\to\infty}n_k=\infty$ and a $\map$-invariant Borel probability measure $\mu$ on $[0,1]$ such that 
$
\frac{1}{n_k}\sum_{j=1}^{n_k}\delta_{\map^j(c)}\to \mu
$
as $k\to\infty$ and the limit 
$\lim_{k\to\infty}\frac{1}{n_k}\log |(\map^{n_k})'(\map(c))|
$
exists. Moreover,  
\[
0=\llya(\map(c))<\lim_{k\to\infty}\frac{1}{n_k}\log |(\map^{n_k})'(\map(c))|\leq \ulya(\map(c)).
\]
    \end{itemize}
\end{thm}
The existence of uncountably many parameters satisfying (1) and (2) follows from Hofbauer and Keller \cite{HofbauerKeller}.
The proof of this theorem is given in Section \ref{sec proof of theorem quadratic}. 
By (1) and (2), for any $a\in \mathcal{S}$ the map $\map$ is non-statistical, that is, for $\Leb$-almost every $x\in [0,1]$ the sequence of measures $\left\{\frac{1}{n}\sum_{k=0}^{n-1}\delta_{\map^k(x)}\right\}_{n\in\mathbb{N}}$ does not converge. In particular, 
the map $\map$ admits neither an absolutely continuous $f_a$-invariant probability measure nor a periodic attractor. (3) of Theorem \ref{thm quadratic map} implies that the Lyapunov exponent along the orbit of $\map(c)$ does not exist.
For $a\in \mathcal{S}$, the statistical behavior of typical orbits gives no indication that the Lyapunov exponent with respect to $\map$ exists for $\Leb$-almost every point. 
Nevertheless, our main theorem shows that for $a\in\mathcal{S}$ the Lyapunov exponent exists and is equal to $0$ for $\Leb$-almost every point. 
Hence, by (1) and (3) of Theorem \ref{thm quadratic map}, for $\Leb$-almost every $x\in [0,1]$ there exists a sequence $\{m_k\}_{k\in\mathbb{N}}\subset \mathbb{N}$ with $\lim_{k\to\infty}m_k=\infty$ such that we have 
\begin{align*}
    &\lim_{k\to\infty}d\left(\frac{1}{m_k}\sum_{j=0}^{m_k-1}\delta_{\map^j(x)},\frac{1}{n_k}\sum_{j=1}^{n_k}\delta_{\map^j(c)}\right)=0 \text{ and }\\&
    \lim_{k\to\infty}\frac{1}{m_k}\log |(\map^{m_k})'(x)|=0<
\lim_{k\to\infty}\frac{1}{n_k}\log |(\map^{n_k})'(\map(c))|,   
\end{align*}
where $\{n_k\}_{k\in\mathbb{N}}$ denotes the sequence appearing in (3) of Theorem \ref{thm quadratic map}. 
This result shows that even when the empirical measures
along a typical orbit and those along the orbit of
$f_a(c)$ converge to the same measure along suitable
subsequences, the corresponding limits of the
finite-time Lyapunov exponents need not coincide.

For the reasons discussed above, we consider the case where an $S$-unimodal map $T$ with a non-flat critical point 
admits neither an absolutely continuous $T$-invariant probability measure nor a periodic attractor. In this case, by a result of Keller \cite{Keller1990}, the upper Lyapunov exponent $\ulya(x)$ is equal to $0$ for $\Leb$-almost every $x\in[0,1]$. However, since Theorem \ref{thm main} requires the existence of the Lyapunov exponent, Keller's result does not by itself imply Theorem \ref{thm main}. To prove Theorem \ref{thm main}, we establish the following result, which is of independent interest. If an $S$-unimodal map $T$ with a non-flat critical point has no periodic attractor then 
\begin{align}\label{eq intro proof st}
\text{
for Lebesgue almost every $x\in [0,1]$ we have $\llya(x)\geq 0$ 
}
\end{align}
(Proposition \ref{prop nonegative}). 
This result establishes a general principle for $S$-unimodal maps: In the absence of periodic attractors, exponential contraction cannot occur along the orbit of $\Leb$-almost every $x\in [0,1]$.

Keller \cite{Keller1990} also obtained a lower bound for $|(T^{n_k})'(x)|$ for $\Leb$-almost every $x\in[0,1]$ along a suitable subsequence $\{n_k\}_{k\in\mathbb N}$. Since this estimate is valid only along this subsequence, it gives a lower bound for the upper Lyapunov exponent $\ulya(x)$ for $\Leb$-almost every $x\in[0,1]$. However, it does not yield a lower bound for the lower Lyapunov exponent $\llya(x)$.
To show \eqref{eq intro proof st}, 
we prove in Proposition \ref{prop slowly recurrent} that if $T$ has no periodic
attractor, then the orbit of Lebesgue almost every
$x\in[0,1]$ satisfies the slowly recurrent condition, that is, for each $\alpha>0$, there exist a constant $C>0$ and $N\in\mathbb{N}$ such that for all $n\geq N$ we have 
    \begin{align}
        |T^n(x)-c|\geq C e^{-\alpha n}.
    \end{align}  
This result allows us to establish \eqref{eq intro proof st}. A result of this type was established by Levin et al. \cite{Lyapunovholomorphic2016} for unicritical complex polynomials whose Julia sets have positive area.
Our proof strategy for Proposition \ref{prop slowly recurrent} is inspired by \cite{Lyapunovholomorphic2016}. More precisely, by using Przytycki \cite[Lemma 1]{characteristic1993}, we show that for any   $\alpha>0$, any $n\in\mathbb{N}$ and any connected component $W$ of $T^{-n}(B(c,e^{-\alpha n}))$ the number of critical points of the map $T^n|_{W}$ is uniformly bounded. 
We then establish a distortion lemma for $T^n|_{W}$ by combining several distortion estimates proved in Section \ref{seq Technical lemmas for distortion estimates of $S$-unimodal maps}: see Claims \ref{claim distorsion} and \ref{claim exponentially small} in the proof of Proposition \ref{prop slowly recurrent}. This enables us to prove Proposition \ref{prop slowly recurrent}. 


We also emphasize that a lower bound for $|(T^{n_k})'(x)|$ along a suitable subsequence $\{n_k\}_{k\in\mathbb N}$, such as the one obtained by Nowicki and Sands \cite{NowickiSands1998} (see Lemma \ref{lemma Nowicki and Sands}), is not sufficient to establish \eqref{eq intro proof st}. Indeed, as explained above, an estimate along a subsequence can provide a lower bound for the upper Lyapunov exponent $\ulya(x)$, but does not by itself provide a lower bound for the lower Lyapunov exponent $\llya(x)$.

The outline of this paper is as follows. 
In Section \ref{seq Technical lemmas for distortion estimates of $S$-unimodal maps}, we establish several distortion lemmas. 
In Section \ref{sec slowly recurrent} we prove that the orbit of Lebesgue almost every $x\in [0,1]$ satisfies the slowly recurrent condition (Proposition \ref{prop slowly recurrent}).
Section \ref{sec known results} introduces known results on $S$-unimodal maps used in Sections \ref{sec proof of theorem main} and \ref{sec proof of theorem quadratic}. 
In Section \ref{sec proof of theorem main}, we prove Theorem \ref{thm main}. In Section \ref{sec proof of theorem quadratic}, we prove Theorem \ref{thm quadratic map}.

\section{Technical lemmas for distortion estimates of $S$-unimodal maps}\label{seq Technical lemmas for distortion estimates of $S$-unimodal maps}
In this section, we introduce several technical lemmas for distortion estimates of $S$-unimodal maps. These lemmas will be used in Section \ref{sec slowly recurrent}.

Let $T$ be an $S$-unimodal map with a non-flat critical point. 
In this paper, we denote by $c$ the critical point of $T$.
For $x\in [0,1]$ we define the $\omega$-limit set $\omega(x)$ of $x$ by 
\[
\omega(x):=\left\{y\in [0,1]: 
\begin{array}{l}
\text{ there exists a sequence }\{n_k\}_{k\in\mathbb{N}} \text{ such that } \\
\lim _{k\to\infty}n_k=\infty
\text{ and }\lim_{k\to\infty} T^{n_k}(x)=y
 \end{array}
 \right\}.
\]
For a periodic point $p\in [0,1]$ with a prime period $q\in \mathbb{N}$ (i.e. $T^q(p)=p$ and $T^k(p)\neq p$ for all $1\leq k\leq q-1$), we define the basin of attraction $B(O(p))$ of the periodic orbit $O(p):=\{p,T(p),\cdots, T^{q-1}(p)\}$ by 
\[
B(O(p)):=\{x\in [0,1]: \omega(x)=O(p)\}.
\]
The map $T$ is said to have a periodic attractor if there exists a periodic point $p\in [0,1]$ such that $B(O(p))$ has non-empty interior. 
Note that any stable periodic orbit is a periodic attractor. Moreover, by \cite[Lemma A.8]{NowickiSands1998}, if there exists a periodic point $p\in [0,1]$ with a prime period $q\in \mathbb{N}$ such that $|(T ^q)'(p)|=1$ then there exists a periodic attractor.

 For an interval $U\subset \mathbb{R}$ we denote by  $|U|$ the Lebesgue measure of $U$. 
For bounded intervals $U\subset V\subset \mathbb{R}$ and $\tau>0$, $V$ is said to contain a $\tau$-scaled neighborhood of $U$ if $V\setminus U$ consists of two components both of length at least $\tau|U|$.
We first introduce the Koebe principle (see, for example, \cite[Chapter IV, Theorem 1.2]{de2012one} or \cite[Lemma A.6]{NowickiSands1998}):
\begin{lemma}\label{lemma Koebe}
    Let $T$ be an $S$-unimodal map and let $\tau>0$. 
    Then there exists a constant $K(\tau)\geq 1$ such that
    for every integer $s\geq 1$, every pair of intervals $I \subset J\subset [0,1]$ such that $T^s|_J$ is a diffeomorphism onto its image and that $T^s(J)$ contains a $\tau$-scaled neighborhood of $T^s(I)$, and for all $x,y\in I$ we have
    \[
    \frac{1}{K(\tau)}\leq \frac{|(T^s)'(x)|}{|(T^s)'(y)|}\leq K(\tau).
    \]
\end{lemma}
Next, we introduce the macroscopic Koebe principle:
\begin{lemma}{\cite[Proposition 1]{BruinRiveraShenStrien}}\label{lemma macroscopic Koebe}
Let $T$ be an $S$-unimodal map with a non-flat critical point such that $T$ has no periodic attractors. Let $\tau>0$. 
Then there exist $K'(\tau)>0$ and $\eta>0$ such that the following condition holds: Let $s\geq 1$ be an integer. Assume that intervals $I \subset J\subset [0,1]$ satisfy the following conditions: 
\begin{itemize}
    \item $|T^s(J)|<\eta$.
    \item $T^s|_J$ is a diffeomorphism onto its image.
    \item $T^{s}(J)$ contains a $\tau$-scaled neighborhood of $T^s(I)$.
\end{itemize}
Then $J$ contains a $K'(\tau)$-scaled neighborhood of $I$. 
\end{lemma}

The following two lemmas give estimates near the critical point.   
Although their proofs use only basic calculus, the lemmas play an important role in Section \ref{sec slowly recurrent}.

\begin{lemma}\label{lemma well inside near critical point}
    Let $T$ be an $S$-unimodal map with a non-flat critical point and let $\tau>0$. 
Then there exist $\delta_c(\tau)>0$ and
$K_c'(\tau)>0$ such that for all
$T(c)-\delta_c(\tau)<y<x<T(c)$ satisfying
\[
x-y\geq\tau(T(c)-x),
\]
the interval $T^{-1}((y,T(c)])$ contains a
$K_c'(\tau)$-scaled neighborhood of
$T^{-1}((x,T(c)])$.
\end{lemma}

\begin{proof}
    Since $T$ has a non-flat critical point, there exist a $C^3$ diffeomorphism $\phi:\mathbb{R}\rightarrow \mathbb{R}$ with $\phi(0)=0$ and real numbers
$l\in(1,\infty)$ and $0<\delta_1<\min\{c,1-c\}$ such that for all $x\in [c-\delta_1,c+\delta_1]$ we have 
    \[
    T(x)=T(c)-|\phi(x-c)|^l.
    \]
    Without loss of generality, we may assume that $\phi$ is monotone increasing. Thus, since $\phi$ is a diffeomorphism, we have 
    \begin{align*}
        a:=\phi'(0)>0.
    \end{align*}
For all $x\in \mathbb{R}$ we define $h(x):=-x+T(c)$ and $k(x):=x+c$. Then for all $x\in [-\delta_1, \delta_1]$ we obtain
\begin{align*}
    \tilde \chi(x):= h\circ T\circ k(x)
    =h(T(x+c))=h(T(c)-|\phi(x)|^l)=|\phi(x)|^l.
\end{align*}
Notice that the statement of the lemma follows once we show that there exist $0<\delta_c(\tau)\leq\min\{\tilde \chi(-\delta_1),\tilde \chi(\delta_1)\}$ and $K_c'(\tau)>0$ such that for each $0<  \tilde x<\tilde y$ satisfying $\tilde y<\delta_c(\tau)$ and $\tilde y-\tilde x\geq \tau \tilde x$
\begin{align}\label{eq critical well inside normal form}
\text{$\tilde \chi^{-1}([0,\tilde y))$ contains a $K_c'(\tau)$-scaled neighborhood of $\tilde \chi^{-1}([0,\tilde x))$.}    
\end{align}
We will show \eqref{eq critical well inside normal form}. 
We define the maps 
\[
\tilde \chi _+:=\tilde \chi|_{[0,\delta_1]}
\text{ and }
\tilde \chi_-:=\tilde \chi|_{[-\delta_1,0]}.
\]
Note that $\tilde \chi_+$ and $\tilde \chi_-$ are injective.
We first show that there exist $0<\delta'<\min\{\tilde \chi(-\delta_1),\tilde \chi(\delta_1)\}$ and $D>0$ such that for all $0\leq z\leq \delta'$ we have 
\begin{align}\label{eq critical important +}
    a^{-1} z^{1/l}-Dz^{2/l}\leq \tilde \chi_+^{-1}(z)\leq a^{-1} z^{1/l}+Dz^{2/l},
\end{align}
and 
\begin{align}\label{eq critical important -}
    -a^{-1} z^{1/l}-Dz^{2/l}\leq \tilde \chi_-^{-1}(z)\leq -a^{-1} z^{1/l}+Dz^{2/l}.
\end{align}
We only show \eqref{eq critical important +}. \eqref{eq critical important -} follows from a similar argument.
The case $z=0$ is trivial, so we may assume that
$0<z\leq\delta'$.
Notice that we have $\tilde \chi_+(x)=h\circ T\circ k(x)=\phi(x)^l$ for all $x\in [0,\delta_1]$. 
Since $\phi$ is in $C^3$ and $\phi(0)=0$,
by shrinking $\delta_1$ if necessary,
we can find a constant $C>0$ such that  
   \begin{align}\label{eq expansion of phi}
        |\phi(x)-ax|\leq C x^{2} \text{ for all } x\in [0,\delta_1].
   \end{align} 
    We define $\tilde \chi_+^*:[0,\delta_1]\rightarrow [0,\infty)$  by 
    \[
    \tilde \chi_{+}^*(x):=(ax)^l.
    \]
    Then, by the mean value theorem, for all $x\in [0,\delta_1]$ we have 
\begin{align}\label{eq critical point approximate}
    |\tilde \chi_+(x)-\tilde \chi_+^*(x)|
    \leq l(|ax|+|\phi(x)-ax|)^{l-1}|\phi(x)-ax|
    \leq C_1x^{l+1},
\end{align}
where $C_1:=Cl(|a|+C\delta_1)^{l-1}$. For $x\in [0,\delta_1]$ we define 
\[
\psi_-(x):=\tilde \chi_+^*(x)-C_1x^{l+1} \text{ and }
\psi_+(x):=\tilde \chi_+^*(x)+C_1x^{l+1}.
\]
By shrinking $\delta_1$ if necessary, for all $x\in (0,\delta_1]$ we have 
\begin{align}\label{eq monotonicity psi}
    \psi_-'(x)>0 \text{ and } \psi_+'(x)>0.
\end{align}
In particular, the maps $\psi_-$ and $\psi_+$ are injective.
Let 
\[
0\leq z\leq
\min\{\tilde\chi(-\delta_1),
\tilde\chi(\delta_1),\psi_-(\delta_1)\}.
\]
Note that $(\tilde \chi_+^{*})^{-1}(z)=a^{-1}z^{1/l}$, and by \eqref{eq critical point approximate}, we have  
\begin{align}\label{eq critical point approximate 2}
|\tilde \chi_+^{-1}(z)-a^{-1}z^{1/l}|\leq \max\{|\psi_+^{-1}(z)-a^{-1}z^{1/l}|,|\psi_-^{-1}(z)-a^{-1}z^{1/l}|\}.
\end{align}
We take a small number $0<\delta'<\min\{\tilde \chi(-\delta_1),\tilde \chi(\delta_1)\}$ so that for all $0\leq z\leq \delta'$ we have 
\[
C_1 (\psi_+^{-1}(z))^{l+1} \leq (a\psi_+^{-1}(z))^{l} \text{ and }C_1 (\psi_-^{-1}(z))^{l+1} \leq 2^{-1}(a\psi_-^{-1}(z))^{l}.
\]
Then since $z=\psi_{\pm}(\psi_{\pm}^{-1}(z))=(a\psi_{\pm}^{-1}(z))^{l}\pm C_1(\psi_{\pm}^{-1}(z))^{l+1}$, for all $0\leq z\leq \delta'$ we obtain 
\begin{align}\label{eq critical point comprability}
   \psi_{+}^{-1}(z) \leq a^{-1}z^{1/l}\leq 2\psi_{+}^{-1}(z)
\text{ and }
   2^{-1}\psi_{-}^{-1}(z) \leq a^{-1}z^{1/l}\leq \psi_{-}^{-1}(z).
\end{align}
On the other hand, since we have $\psi_{\pm}(a^{-1}z^{1/l})-\psi_{\pm}(\psi_{\pm}^{-1}(z))=\pm C_1 a^{-(l+1)}z^{1+1/l}$, the mean value theorem implies that there exist two points $\xi_\pm$ between $a^{-1}z^{1/l}$ and $\psi_{\pm}^{-1}(z)$ such that 
\[
|C_1 a^{-(l+1)}z^{1+1/l}|=|\psi'_+(\xi_+)||a^{-1}z^{1/l}-\psi^{-1}_+(z)|=|\psi'_-(\xi_-)||a^{-1}z^{1/l}-\psi^{-1}_-(z)|.
\]
By shrinking $\delta'$ if necessary, for all $0\leq z\leq \delta'$ we have $2^{-1}la> C_1(l+1)a^{-l}z^{1/l}$.  
Since $\psi_\pm'(\xi_\pm)=l(a\xi_\pm)^{l-1}a\pm C_1(l+1)\xi_\pm^{l}$, \eqref{eq critical point comprability} yields that there exists a constant $C_2\geq1$ such that for all $0\leq z\leq \delta'$ we have 
\begin{align*}
    &\frac{1}{C_2}
    \left|\frac{C_1a^{-(l+1)}z^{2/l}}{la+ C_1(l+1)a^{-l}z^{1/l}}\right|
    \leq \left|\frac{z^{1/l}}{a}-\psi^{-1}_+(z)\right|
    \leq C_2\left|\frac{C_1a^{-(l+1)}z^{2/l}}{la+ C_1(l+1)a^{-l}z^{1/l}}\right| \text{ and }
    \\&\frac{1}{C_2}
    \left|\frac{C_1a^{-(l+1)}z^{2/l}}{la- C_1(l+1)a^{-l}z^{1/l}}\right|
    \leq \left|\frac{z^{1/l}}{a}-\psi^{-1}_-(z)\right|
    \leq C_2\left|\frac{C_1a^{-(l+1)}z^{2/l}}{la- C_1(l+1)a^{-l}z^{1/l}}\right|.
\end{align*}
Therefore, since $1<l<\infty$, there exists a constant $D\geq 1$ such that for all $0\leq z\leq \delta'$ we have 
\begin{align*}
    &\frac{z^{2/l}}{D}
    \leq \left|\frac{z^{1/l}}{a}-\psi^{-1}_+(z)\right|
    \leq Dz^{2/l}
    \text{ and }\frac{z^{2/l}}{D}
    \leq \left|\frac{z^{1/l}}{a}-\psi^{-1}_-(z)\right|
    \leq Dz^{2/l}.
\end{align*}
Combining this with \eqref{eq critical point approximate 2}, we obtain \eqref{eq critical important +}. Hence, the proof of \eqref{eq critical important +} and \eqref{eq critical important -} is complete.

For two numbers $0< \tilde x<\tilde y$ satisfying 
$\tilde y<\delta'$ and 
$\tilde y-\tilde x\geq \tau \tilde x$
we set
\[
\tilde J:=[0,\tilde y) \text{ and }\tilde I:=[0,\tilde x).
\]
$\tilde L ^{-1}_{\tilde J,\tilde I}$ and $\tilde R^{-1}_{\tilde J,\tilde I}$ denote the two connected components of $\tilde  \chi^{-1}(\tilde J\setminus\tilde I)$ satisfying $\tilde L ^{-1}_{\tilde J,\tilde I}\subset (-\infty,0)$ and $\tilde R^{-1}_{\tilde J,\tilde I}\subset (0,\infty)$.  
By \eqref{eq critical important +} and \eqref{eq critical important -}, we have 
\begin{align}\label{eq critical first stap}
    \frac{|\tilde L_{\tilde J,\tilde I}^{-1}|}{|\tilde \chi^{-1}(\tilde I)|}\geq \frac{(\tilde y^{1/l}-\tilde x^{1/l})-aD(\tilde y^{2/l}+\tilde x^{2/l})}{2\tilde x^{1/l}+2Da\tilde x^{2/l}}.
\end{align}
There exists $0<\delta''<\delta'$ such that for all $0\leq z<\delta''$ we have 
\begin{align}\label{eq critical point weak}
    z^{2/l}\leq (4aD)^{-1}z^{1/l}.
\end{align}
Let $0< \tilde x<\tilde y$ satisfy
$\tilde y<\delta''$ and 
$\tilde y-\tilde x\geq \tau \tilde x$.
We first consider the case where $\tilde y/\tilde x\geq 4^l$. By \eqref{eq critical first stap} and \eqref{eq critical point weak}, we have 
\begin{align}\label{eq critical point easy case}
    \frac{|\tilde L_{\tilde J,\tilde I}^{-1}|}{|\tilde \chi^{-1}(\tilde I)|}\geq 
    \frac{2^{-1}\tilde y^{1/l}-\tilde x^{1/l}}{2\tilde x^{1/l}+2Da\tilde x^{2/l}}=
    \frac{2^{-1}(\tilde y/\tilde x)^{1/l}-1}{2+2Da\tilde x^{1/l}}\geq 
    \frac{1}{2+2Da}.
\end{align}
Next, we consider the case where $1+\tau \leq \tilde y/\tilde x\leq 4^l$. By \eqref{eq critical first stap}, we have 
\begin{align}\label{eq critical point difficult case}
    &\frac{|\tilde L_{\tilde J,\tilde I}^{-1}|}{|\tilde \chi^{-1}(\tilde I)|}
    \geq 
    \frac{((\tilde y/\tilde x)^{1/l}-1)-aD(\tilde y^{1/l}(\tilde y/\tilde x)^{1/l}+\tilde x^{1/l})}{2+2Da\tilde x^{1/l}}
    \\&\geq\frac{(1+\tau)^{1/l}-1}{2+2Da}-\frac{aD(\tilde y^{1/l}(\tilde y/\tilde x)^{1/l}+\tilde x^{1/l})}{2+2Da\tilde x^{1/l}}.\nonumber
\end{align}
Since $\tilde y/\tilde x\leq 4^l$ we obtain
\begin{align*}
    \left|\frac{aD(\tilde y^{1/l}(\tilde y/\tilde x)^{1/l}+\tilde x^{1/l})}{2+2Da\tilde x^{1/l}}\right|\leq 
    \frac{aD(4\tilde y^{1/l}+\tilde x^{1/l})}{2}\leq \frac{aD(4+1)}{2}\tilde y^{1/l}. 
\end{align*}
Therefore, there exists $0<\delta:=\delta_c(\tau)<\delta''$ such that for all $0< \tilde x<\tilde y$ satisfying
$\tilde y<\delta$ and 
$1+\tau \leq \tilde y/\tilde x\leq 4^l$ we obtain 
\[
\left|\frac{aD(\tilde y^{1/l}(\tilde y/\tilde x)^{1/l}+\tilde x^{1/l})}{2+2Da\tilde x^{1/l}}\right|\leq 
\frac{1}{2} \frac{(1+\tau)^{1/l}-1}{2+2Da}.
\]
 Combining this with \eqref{eq critical point difficult case}, for all $0< \tilde x<\tilde y$ satisfying
$\tilde y<\delta$ and 
$1+\tau \leq \tilde y/\tilde x\leq 4^l$ we obtain 
\[
\frac{|\tilde L_{\tilde J,\tilde I}^{-1}|}{|\tilde \chi^{-1}(\tilde I)|}\geq \frac{1}{2} \frac{(1+\tau)^{1/l}-1}{2+2Da}.
\]
By a similar argument, we can show that 
\[
\frac{|\tilde R_{\tilde J,\tilde I}^{-1}|}{|\tilde \chi^{-1}(\tilde I)|}\geq \frac{1}{2} \frac{(1+\tau)^{1/l}-1}{2+2Da}.
\]
By setting $K_c'(\tau):=\min\{ \frac{(1+\tau)^{1/l}-1}{4+4Da}, \frac{1}{2+2Da}\}$ and using $\eqref{eq critical point easy case}$ and this inequality, we obtain \eqref{eq critical well inside normal form}. Thus, we are done.
\end{proof}

 Since $[0,1]$ is compact and $T$ is continuous on $[0,1]$, if $T$ has a critical point with order $l$ then there exists a constant $W\geq 1$ such that for all
 $x\in (0,1)\setminus\{c\}$ we have
\begin{align}\label{eq critical order}
    \frac{1}{W}\leq \frac{|T(x)-T(c)|}{|x-c|^{l}}\leq W  \text{ and } \frac{1}{W}\leq \frac{|T'(x)|}{|x-c|^{l-1}}\leq W.
\end{align}

\begin{lemma}\label{lemma ratio}
    Let $T$ be an $S$-unimodal map with a non-flat critical point and let $l$ be the critical order.  Then there exist $\eta>0$ and $C>0$ such that the following conditions hold:
    \begin{itemize}
        \item[(D1)] For any intervals $I,J\subset [0,1]$ with $c\in I\subset J$ we have 
    \[
    \frac{|I|}{|J|}\leq C\left(\frac{|T(I)|}{|T(J)|}\right)^{1/l}.
    \]
    \item[(D2)] For each $0\leq y_1<x_1<x_2<y_2<T(c)$ with $T(c)-y_1<\eta$, a connected component $I$ of $T^{-1}((x_1,x_2))$ and a connected component $J$ of $T^{-1}((y_1,y_2))$ with $I\subset J$ we have  
    \begin{align*}
        \frac{|I|}{|J|}\leq C\left(\frac{|(x_1,x_2)|}{|(y_1,y_2)|}\right)^{1/l}.
    \end{align*}
    \end{itemize}
\end{lemma}

\begin{proof}
 (D1) follows from \eqref{eq critical order}. Therefore, we will show (D2).
    Since $T$ has a critical point with order $l$, 
    there exist a $C^3$ diffeomorphism $\phi:\mathbb{R}\rightarrow \mathbb{R}$ with $\phi(0)=0$ and a real number $0<\delta_1<\min\{c,1-c\}$ such that for all $x\in [c-\delta_1,c+\delta_1]$ we have 
    $
    T(x)=T(c)-|\phi(x-c)|^l.
    $
    As in the proof of Lemma \ref{lemma well inside near critical point}, we may assume that $\phi$ is monotone increasing. We set 
    $
        a:=\phi'(0)>0.
    $
    For $x\in [-\delta_1,\delta_1]$ we set $\tilde \chi(x):=|\phi(x)|^l$.
It is enough to show that 
there exist $0<\eta\leq \min\{\tilde \chi(-\delta_1),\tilde \chi(\delta_1)\}$ and $C>0$ such that 
for each $0< \tilde y_2<\tilde x_2<\tilde x_1<\tilde y_1 <\eta$, a connected component $\tilde I$ of $\tilde \chi^{-1}((\tilde x_2,\tilde x_1))$ and a connected component $\tilde J$ of $\tilde \chi^{-1}((\tilde y_2,\tilde y_1))$ with $\tilde I\subset \tilde J$ we have 
\[
    \frac{|\tilde I|}
    {|\tilde J|}
    \leq C\left(\frac{|(\tilde x_2,\tilde x_1)|}{|(\tilde y_2,\tilde y_1) |}\right)^{1/l}.
\]
We define $\tilde \chi_+:=\tilde \chi|_{[0,\delta_1]}$ and $\tilde \chi_-:=\tilde \chi|_{[-\delta_1,0]}$. By \eqref{eq critical important +} and \eqref{eq critical important -}, 
there exist $0<\eta'<\min\{\tilde \chi(-\delta_1),\tilde \chi(\delta_1)\}$ and $D>0$ such that for all $0\leq z\leq \eta'$ we have 
\begin{align}\label{eq important restate}
   & a^{-1} z^{1/l}-Dz^{2/l}\leq \tilde \chi_+^{-1}(z)\leq a^{-1} z^{1/l}+Dz^{2/l}
\text{ and } 
    \\&-a^{-1} z^{1/l}-Dz^{2/l}\leq \tilde \chi_-^{-1}(z)\leq -a^{-1} z^{1/l}+Dz^{2/l}. \nonumber
\end{align}
 Let $0< \tilde y_2<\tilde x_2<\tilde x_1<\tilde y_1 <\eta'$. 
 By \eqref{eq critical order}, there exists a constant $D_1\geq 1$ such that for all $z\in (-\delta_1,\delta_1)$ we have 
\begin{align}\label{eq derivatibe comperability}
    D^{-1}_1|z|^{l-1}\leq |\tilde \chi'(z)| \leq D_1 |z|^{l-1}.
\end{align}
We first consider the case where $\tilde y_1/\tilde y_2\leq 2$. 
By \eqref{eq important restate} and by shrinking $\eta'$ if necessary, we can find a constant $C\geq 1$ such that for all $0\leq z\leq \eta'$ we have $C^{-1}z^{1/l}\leq |\tilde \chi_{\pm}^{-1}(z)|\leq Cz^{1/l}$.
Hence, by \eqref{eq derivatibe comperability},
for all $z\in [\tilde y_2,\tilde y_1]$ we have $(2D_1)^{-1}C^{-(l-1)}\tilde y_2^{1/l-1}\leq |(\tilde \chi_\pm^{-1})'(z)|\leq C^{l-1}D_1\tilde y_2^{1/l-1}$. Thus, by the mean value theorem, we have 
\begin{align}\label{eq ratio D2 case 1}
    \frac{|\tilde \chi_\pm^{-1}((\tilde x_2,\tilde x_1))|}{|\tilde \chi_\pm^{-1}((\tilde y_2,\tilde y_1))|}
    \leq 2D_1^2C^{2(l-1)} \frac{|(\tilde x_2,\tilde x_1)|}{|(\tilde y_2,\tilde y_1)|}\leq  
    2D_1^2 C^{2(l-1)}\left(\frac{|(\tilde x_2,\tilde x_1)|}{|(\tilde y_2,\tilde y_1)|}\right)^{1/l}.
\end{align}

Since $\lim_{t\to\infty}(t^{1/l}-1)/(t-1)^{1/l}=1$, there exists a constant $D_2\geq 1$ such that for all $0< \tilde z_2<\tilde z_1 <\eta'$ with $\tilde z_1> 2 \tilde z_2$ we have
\begin{align}\label{eq ratio comperability y}
   {D_2}^{-1}(\tilde z_1-\tilde z_2)^{1/l} \leq {\tilde z^{1/l}_1-\tilde z^{1/l}_2}{}\leq D_2(\tilde z_1-\tilde z_2)^{1/l}.
\end{align}
Moreover, for all $0< \tilde z_2<\tilde z_1 <\eta'$ with $\tilde z_1> 2 \tilde z_2$ we have 
\begin{align*}
    \frac{\tilde z_1^{2/l}+\tilde z_2^{2/l}}
    {\tilde z_1^{1/l}-\tilde z_2^{1/l}}
    \leq\tilde z_1^{1/l}\frac{(\tilde z_1/\tilde z_2)^{1/l}}{(\tilde z_1/\tilde z_2)^{1/l}-1}+ \frac{\tilde z_2^{1/l}}{2^{1/l}-1}. 
\end{align*}
Since the function $t\in [2,\infty)\mapsto t^{1/l}/(t^{1/l}-1)$ is bounded, by shrinking $\eta'$ if necessary, for all $0< \tilde z_2<\tilde z_1 <\eta'$ with $\tilde z_1> 2 \tilde z_2$ we obtain
\begin{align}\label{eq ratio error}
    \tilde z_1^{2/l}+\tilde z_2^{2/l}\leq (2aD)^{-1}(\tilde z_1^{1/l}-\tilde z_2^{1/l}).
\end{align}
Thus, by \eqref{eq important restate} and \eqref{eq ratio comperability y}, for all $0< \tilde z_2<\tilde z_1 <\eta'$ with $\tilde z_1> 2 \tilde z_2$ we have
\begin{align}\label{eq error 1}
    &|\tilde \chi_\pm^{-1}((\tilde z_2,\tilde z_1))|\geq a^{-1}(\tilde z_1^{1/l}-\tilde z_2^{1/l})-D(\tilde z_1^{2/l}+\tilde z_2^{2/l})
    \\&\geq (2a)^{-1}(\tilde z_1^{1/l}-\tilde z_2^{1/l})
    \geq (2aD_2)^{-1}(\tilde z_1-\tilde z_2)^{1/l} \nonumber
\end{align}
and 
\begin{align}\label{eq error 2}
    &|\tilde \chi_\pm^{-1}((\tilde z_2,\tilde z_1))|\leq a^{-1}(\tilde z_1^{1/l}-\tilde z_2^{1/l})+D(\tilde z_1^{2/l}+\tilde z_2^{2/l})
    \\&\leq 2a^{-1}(\tilde z_1^{1/l}-\tilde z_2^{1/l})
    \leq 2a^{-1}D_2(\tilde z_1-\tilde z_2)^{1/l}. \nonumber
\end{align}

Next, we consider the case where $\tilde y_1/\tilde y_2> 2$ and $\tilde x_1/\tilde x_2\leq 2$.
 By \eqref{eq derivatibe comperability}, for all $z\in [\tilde x_2,\tilde x_1]$ we have 
 $(2D_1)^{-1}C^{-(l-1)}\tilde x_2^{1/l-1}\leq |(\tilde \chi_\pm^{-1})'(z)|\leq C^{l-1}D_1\tilde x_2^{1/l-1}$.
 Thus, by the mean value theorem and \eqref{eq error 1}, we have 
\begin{align}\label{eq ratio D2 case 2}
    &\frac{|\tilde \chi_\pm^{-1}((\tilde x_2,\tilde x_1))|}{|\tilde \chi_\pm^{-1}((\tilde y_2,\tilde y_1))|}
    \leq D_3
     \frac{\tilde x_2^{1/l-1}|(\tilde x_2,\tilde x_1)|^{1-1/l}|(\tilde x_2,\tilde x_1)|^{1/l}}{|(\tilde y_2,\tilde y_1)|^{1/l}}
    \\&\leq 3D_3 
     \left(\frac{|(\tilde x_2,\tilde x_1)|}{|(\tilde y_2,\tilde y_1)|}\right)^{1/l}, \nonumber
\end{align}
where $D_3:=2a D_2D_1C^{l-1}$.
If $\tilde y_1/\tilde y_2> 2$ and $\tilde x_1/\tilde x_2> 2$ then by \eqref{eq error 1} and \eqref{eq error 2}, we obtain 
\[
\frac{|\tilde \chi_\pm^{-1}((\tilde x_2,\tilde x_1))|}{|\tilde \chi_\pm^{-1}((\tilde y_2,\tilde y_1))|}
    \leq 4D_2^2 
     \left(\frac{|(\tilde x_2,\tilde x_1)|}{|(\tilde y_2,\tilde y_1)|}\right)^{1/l}.
\]
Combining this with \eqref{eq ratio D2 case 1} and \eqref{eq ratio D2 case 2}, we obtain (D2) and the proof is complete.
\end{proof}

\section{Slowly recurrent condition}\label{sec slowly recurrent}

In this section, we prove that if $T$ has no periodic
attractor, then the orbit of Lebesgue almost every
$x\in[0,1]$ satisfies the slowly recurrent condition.
\begin{definition}\label{def slowly recurrent}
    Let $T$ be an $S$-unimodal map and let $x\in [0,1]$. 
    The orbit $O(x):=\{T^n(x):n\in\mathbb{N}\}$  of $x$ is said to satisfy the slowly recurrent condition if 
    for each $\alpha>0$, there exist a constant $C>0$ and $N\in\mathbb{N}$ such that for all $n\geq N$ we have 
    \begin{align}\label{eq slowly recurrent}
        |T^n(x)-c|\geq C e^{-\alpha n}.
    \end{align}
\end{definition}

The following two lemmas are due to Przytycki \cite{characteristic1993} (see \cite[Section 3]{characteristic1993} for the proof of the real version).

\begin{lemma}{\cite[Lemma 1 and Section 3]{characteristic1993}} \label{lemmma Ruelle's rule}
    Let $T$ be an $S$-unimodal map with a non-flat critical point. We assume that $T$ has no periodic attractors. Then there exists a constant $C>0$ such that for each $\epsilon>0$ and $n\in\mathbb{N}$ with $T^{n}(B(c,\epsilon))\cap B(c,\epsilon)\neq \emptyset$ we have
    \begin{align*}
        n\geq -C\log \epsilon.
    \end{align*}
\end{lemma}

\begin{lemma}{\cite[Lemma 3 and Section 3]{characteristic1993}}\label{lemma small ball}
    Let $T$ be an $S$-unimodal map with a non-flat critical point. We assume that $T$ has no periodic attractors.     
    Then for every $\epsilon>0$,
$0<k<1$, and $M>0$, there exists
$\delta_0>0$ such that for every
$0<\delta\leq\delta_0$, every $n\in\mathbb{N}$,
and every connected component $V$ of
$T^{-n}(B(c,\delta))$ containing at most $M$
critical points of $T^n$, we have
$|W|<\epsilon$ for every connected component
$W$ of $T^{-n}(B(c,k\delta))$ with $W\subset V$.
\end{lemma}

The following proposition is the main result of this section.
\begin{prop}\label{prop slowly recurrent}
    Let $T$ be an $S$-unimodal map with a non-flat critical point. We assume that $T$ has no periodic attractors. 
    Then for $\Leb$-almost every $x\in [0,1]$ the orbit $O(x)$ of $x$ satisfies the slowly recurrent condition.
\end{prop}

\begin{proof}
Let $\alpha>0$. For $n\in\mathbb{N}$ we define
\[
B_n^1:=B(c,e^{-\alpha n}).
\]
We denote by $C$ the constant obtained in Lemma \ref{lemmma Ruelle's rule}.
Let $n\in \mathbb{N}$ and let $W$ be a connected component of $T^{-n}(B_n^1)$. Suppose that $T^{j_1}(W)$ and $T^{j_2}(W)$ both contain $c$ for some $0\leq j_1<j_2\leq n$. Then there exist $x_1,x_2\in W$ such that $T^
{j_1}(x_1)=T^{j_2}(x_2)=c$. Note that since $x_1,x_2\in W$, for $i\in \{1,2\}$ we have
\[
T^{n-j_i}(c)=T^{n-j_i}(T^{j_i}(x_i))=T^n(x_i)\in B_n^1
\text{ and }
T^{j_2-j_1}(T^{n-j_2}(c))=T^{n-j_1}(c).
\]
Hence, by applying Lemma \ref{lemmma Ruelle's rule} with $n=j_2-j_1$ and $\epsilon=e^{-\alpha n}$, we have 
\[
j_2-j_1\geq \alpha Cn.
\]
Thus, we obtain
\begin{align}\label{eq finiteness}
\#\{0\leq j\leq n: c\in T^j(W)\}\leq M:=2+\lceil(\alpha C)^{-1}\rceil. 
\end{align}
Let $N\in\mathbb{N}$ be a large number such that for all $n\geq N$ we have $B_n^1\subset (0,1)$. Let $n\geq N$ and let $B\subset B_n^1$ be an open ball centered at $c$. Let $W$ be a connected component of $T^{-n}(B)$.
We define the sequence $\{V_{m,B}(W)\}_{0\leq m\leq n}$ of intervals as follows: We set $V_{0,B}(W):=B$. For $1\leq m\leq n$ we define $V_{m,B}(W)$ to be the unique connected component of $T^{-m}(B)$ satisfying $T^{n-m}(W)\subset V_{m,B}(W)$.
We also define the sequence 
\begin{align*}
    &0=:k_0^{B,W}< k_1^{B,W}<k_2^{B,W}<\cdots<k_{p(B,W)}^{B,W}<k_{p(B,W)+1}^{B,W}:= n+1
\end{align*}
as follows:
We set $k_0^{B,W}:=0$.
If there exists $1\leq k\leq n$ such that $c\in V_{k,B}(W)$ then we define $k_1^{B,W}$ to be the smallest number $1\leq k\leq n$ satisfying $c\in V_{k,B}(W)$. 
If there is no $1\leq k\leq n$ such that $c\in V_{k,B}(W)$ then we set $k_{1}^{B,W}:=n+1$ and $p(B,W):=0$. 
Suppose that $k_s^{B,W}$ has already been defined and that $k_s^{B,W}< n+1$ for some $s\in\mathbb{N}$.  
If there exists $k_s^{B,W}< k\leq n$ such that $c\in V_{k,B}(W)$ then we define $k_{s+1}^{B,W}$ to be the smallest number $k_s^{B,W}< k\leq n$ satisfying $c\in V_{k,B}(W)$. 
If there is no $k_s^{B,W}< k\leq n$ such that $c\in V_{k,B}(W)$ then we set $k_{s+1}^{B,W}:=n+1$ and $p(B,W):=s$. 

Note that for all $1\leq j\leq p(B,W)+1$ with $k_j^{B,W}-k_{j-1}^{B,W}-1>0$ 
\begin{align}\label{eq bijection}
\text{the map }
T^{k_j^{B,W}-k_{j-1}^{B,W}-1}:V_{k_j^{B,W}-1,B}(W)\rightarrow V_{k_{j-1}^{B,W},B}(W) 
\text{ is a bijection}.
\end{align}
Moreover, by the same argument as that used to obtain \eqref{eq finiteness}, we have 
\begin{align}\label{eq finiteness 2}
    p (B,W)\leq M.
\end{align}
For every $n\geq N$, every pair of open balls $B,B'$ centered at $c$ with $B\subset B'\subset B_n^1$ and every connected component $W$ of $T^{-n}(B)$ there exists a unique connected component $\text{Inc}_{B,B'}(W)$ of $T^{-n}(B')$ such that $W\subset \text{Inc}_{B,B'}(W)$. 
By \eqref{eq finiteness}, for all $n\geq N$
\begin{align}\label{eq finite to 1}
    \text{the map $W\mapsto \text{Inc}_{B,B'}(W)$  is at most $2^M$-to-$1$.}
\end{align}
For $n\in\mathbb{N}$ we define
\begin{align*}
    &B_n^2:=B(c,e^{-\alpha (n+1)}),\ B_n^{3}:=B(c,e^{-2\alpha n})
    \text{ and }
    B_n^{4}:=B(c,e^{-4\alpha n}).
\end{align*}
For each $i\in \{1,2,3,4\}$ and $n\in \mathbb{N}$ we set
\begin{align*}
\mathcal{C}_n^{i}:=\{W^i \subset [0,1]:W^i \text{ is a connected component of }T^{-n}(B_n^i)\}.
\end{align*} 
For all $0\leq m\leq n$, $i\in \{1,2,3,4\}$ and $W^i\in \mathcal{C}_n^i$ we set $V_{m,i}(W^i):=V_{m,B_n^i}(W^i)$, $p(W^i):=p(B_n^{i}, W^i)$ and $k_{j}^{W^i}:=k_j^{B_n^i,W^i}$ $(0\leq j\leq p(W^i)+1)$.

For $n\geq N$ and $0\leq j\leq M+3$ we set
\begin{align*}
    B_{n,j}^2:=B(c, e^{- \alpha n -\alpha+j\alpha /(M+3)}).
\end{align*}
Then, by the same argument as that used to obtain \eqref{eq finiteness}, for each $n\geq N$ there exists $0\leq q(n)\leq M$ such that 
\begin{align}\label{eq decomposition no critical value}
(B_{n,q(n)+1}^2\setminus B_{n,q(n)}^2 )\cap \bigcup_{k=1}^{n} \{T^k(c)\}=\emptyset.    
\end{align}
For each $n\geq N$ and a connected component $\tilde W$ of $T^{-n}(B_{n,q(n)}^2)$ we set
\begin{align*}
    \tilde B_n:=B_{n,q(n)}^2,\ \tilde B_n':=B_{n,q(n)+1}^2 \text{ and }
    \tilde W':=\text{Inc}_{\tilde B_n,\tilde B'_n}(\tilde W).
\end{align*}
Then for all $n\geq N$ and a connected component $\tilde W$ of $T^{-n}(\tilde B_{n})$ we have 
\begin{align*}
    &p(\tilde W):=p(\tilde B_{n},\tilde W)=p(\tilde B_{n}',\tilde W') \text{ and }
k_j^{\tilde W}:=k_{j}^{\tilde B_{n},\tilde W}=k_{j}^{\tilde B_{n}',\tilde W'}. 
\end{align*}
For each $n\geq N$ we denote by $\tilde {\mathcal{C}}_n$ the family of connected components of $T^{-n}(\tilde B_n)$.
\begin{claim}\label{claim distorsion}
  There exist $\tilde N\geq N$ and $C_1\geq 1$ such that for all $n\geq  \tilde N$, $\tilde W\in \tilde{\mathcal{C}}_n$, $1\leq j\leq p(\tilde W)+1$  with $k_j^{\tilde W}-k_{j-1}^{\tilde W}-1>0$ and $x,y\in V_{k_j^{\tilde W}-1,\tilde B_n}(\tilde W)$ we have 
  \begin{align*}
      \frac{1}{C_1}\leq \frac{|(T^{k_j^{\tilde W}-k_{j-1}^{\tilde W}-1})'(x)|}{|(T^{k_j^{\tilde W}-k_{j-1}^{\tilde W}-1})'(y)|}\leq C_1.
  \end{align*}
\end{claim}
\emph{Proof of Claim \ref{claim distorsion}.}
By the Koebe principle (Lemma \ref{lemma Koebe}) and \eqref{eq bijection}, the claim follows once we show that there exist $\tilde N\geq N$ and $\tau>0$ such that for every $n\geq \tilde N$, every connected component $\tilde W$ of $T^{-n}(\tilde B_{n})$ and every $1\leq j\leq p(\tilde W)+1$ 
\begin{align}\label{eq scaled neighborhood}
\text{ 
$V_{k_{j-1}^{\tilde W}}(\tilde W')$
 contains  
a $\tau$-scaled neighborhood of 
$
V_{k_{j-1}^{\tilde W}}(\tilde W),
$
}
\end{align}
where $V_{k_{j-1}^{\tilde W}}(\tilde W'):=V_{k_{j-1}^{\tilde W}, \tilde B_n'}(\tilde W')$ and $V_{k_{j-1}^{\tilde W}}(\tilde W):=V_{k_{j-1}^{\tilde W}, \tilde B_n}(\tilde W)$.
We will show this by induction. Let 
$
p:=\sup_{n\geq N} \max\{p(\tilde W):\tilde W\in \tilde {\mathcal{C}}_n\}.
$
Then, by \eqref{eq finiteness 2}, we have $p\leq M$.
For all $n\geq N$ and $\tilde W\in \tilde {\mathcal{C}}_n$
we have $V_{k_{0}^{\tilde W}}(\tilde W')=\tilde B'_n$
and 
$
V_{k_{0}^{\tilde W}}(\tilde W)=\tilde B_n.
$
In particular, $V_{k_{0}^{\tilde W}}(\tilde W')$ contains a $(e^{\alpha/(M+3)}-1)/2$-scaled neighborhood of $V_{k_{0}^{\tilde W}}(\tilde W)$. We set 
\[
N_1:=N \text{ and }\tau_1:=(e^{\alpha/(M+3)}-1)/2.
\]
Let $1\leq s\leq p$.
Suppose that $N_s\geq N$ and $\tau_s>0$ have already been defined so that, for every $n\ge N_s$, every $\tilde W\in \tilde {\mathcal{C}}_n$, and every $1\le j\le \min\{s,p(\tilde W)+1\}$, the interval $V_{k_{j-1}^{\tilde W}}(\tilde W')$ contains a $\tau_s$-scaled neighborhood of $V_{k_{j-1}^{\tilde W}}(\tilde W)$. 
Let $\eta>0$ denote the number obtained in the macroscopic Koebe principle (Lemma \ref{lemma macroscopic Koebe}). By using \eqref{eq finiteness} and applying Lemma \ref{lemma small ball}
with $\epsilon=\eta$, 
we can find $N'_s\geq N_s$ such that  for all
$n\geq N_s'$, every $\tilde W\in \tilde {\mathcal{C}}_n$, and every $0\leq m\leq n$, we have 
\begin{align}\label{eq small preimage}
|V_{m,\tilde B_n'}(\tilde W')|<\eta.    
\end{align}
We now consider $n\geq N'_s$ and
$\tilde W\in\tilde{\mathcal{C}}_n$
with $p(\tilde W)\geq s$.
Equations \eqref{eq bijection} and \eqref{eq small preimage}, together with the induction hypothesis, allow us to apply the macroscopic Koebe principle (Lemma \ref{lemma macroscopic Koebe}). Namely, let $K'(\tau_s)$ be the constant obtained in Lemma \ref{lemma macroscopic Koebe}.
Then 
\begin{align}\label{eq induction}
\text{ $V_{k_{s}^{\tilde W}-1}(\tilde W')$ contains a $K'(\tau_s)$-scaled neighborhood of $V_{k_{s}^{\tilde W}-1}(\tilde W)$.}    
\end{align}
Using Lemma \ref{lemma small ball}, by enlarging $N_s'$ if necessary, we may assume that  for all
$n\geq N_s'$, every $\tilde W\in \tilde{\mathcal{C}}_n$, and every $0\leq m\leq n$,
we have 
\begin{align}\label{eq more small}
|V_{m,\tilde B_n'}(\tilde W')|<\delta_c(K'(\tau_s)),
\end{align}
where $\delta_c(K'(\tau_s))$ denotes the small number obtained in Lemma \ref{lemma well inside near critical point}.    
Since $c\in V_{k_{s}^{\tilde W}}(\tilde W)$ by the construction, we have 
\begin{align*}
&T^{-1}(V_{k_{s}^{\tilde W}-1}(\tilde W))=T^{-1}((x,T(c)])=V_{k_{s}^{\tilde W}}(\tilde W)
\text{ and }
\\&
T^{-1}(V_{k_{s}^{\tilde W}-1}(\tilde W'))=T^{-1}((y,T(c)])=V_{k_{s}^{\tilde W}}(\tilde W'), 
\end{align*}
where $V_{k_{s}^{\tilde W}-1}(\tilde W)=(x,z)$ and $V_{k_{s}^{\tilde W}-1}(\tilde W')=(y,z')$. 
Note that by \eqref{eq induction} and \eqref{eq more small}, we have $x-y\geq K'(\tau_s)(T(c)-x)$ and $T(c)-\delta_c(K'(\tau_s))<y<x<T(c)$. 
We denote by $K_c'(K'(\tau_s))$ the constant obtained in Lemma \ref{lemma well inside near critical point}. 
By Lemma \ref{lemma well inside near critical point},
for all $n\geq N'_s$ and
$\tilde W\in \tilde{\mathcal{C}}_n$
with $p(\tilde W)\geq s$,
$V_{k_{s}^{\tilde W}}(\tilde W')$ contains
a $K'_c(K'(\tau_s))$-scaled neighborhood of
$V_{k_{s}^{\tilde W}}(\tilde W)$.
We set $N_{s+1}:=N'_s$ and $\tau _{s+1}:=\min\{K'_c(K'(\tau_s)),\tau_s\}$.  

Hence, by induction, we obtain $N_{p+1}$ and $\tau_{p+1}$ such that for all $n\ge N_{p+1}$, every $\tilde W\in \tilde {\mathcal{C}}_n$, and every $1\le j\le p(\tilde W)+1$, the interval $V_{k_{j-1}^{\tilde W}}(\tilde W')$ contains a $\tau_{p+1}$-scaled neighborhood of $V_{k_{j-1}^{\tilde W}}(\tilde W)$.
By setting $\tilde N:=N_{p+1}$ and $\tau:=\tau_{p+1}$, we complete the proof of \eqref{eq scaled neighborhood}. Hence, the proof of Claim \ref{claim distorsion} is also complete. 
\qed

For all $n\geq N$, $W^4\in \mathcal{C}_n^{4}$ and $i\in\{1,2,3\}$ we set $\text{Inc}_i(W^4):=\text{Inc}_{B_n^4,B_n^{i}}(W^4)$. Next, we will show the following claim:
\begin{claim}\label{claim exponentially small}
    There exist $N_1\geq N$ and $C_3>0$ such that for all $n\geq N_1$ and $W^4\in \mathcal{C}_n^{4}$ we have 
    \begin{align*}
        \frac{|W^4|}{|\text{Inc}_{3}(W^4)|}\leq C_3 e^{-\frac{\alpha n}{2(M+1)l^{M}}}, 
    \end{align*}
    where $l$ denotes the order of the critical point $c$.
\end{claim}
\emph{Proof of Claim \ref{claim exponentially small}.}
We denote by $\eta$ the constant obtained in Lemma \ref{lemma ratio} and by $\tilde N$ the natural number obtained in Claim \ref{claim distorsion}. By Lemma \ref{lemma small ball}, there exists $N_0\geq \tilde N$ such that for all $n\geq N_0$, $\tilde W \in \tilde {\mathcal{C}}_n$ and $0\leq m\leq n$ we have 
\begin{align}\label{eq small preimage claim 2}
|V_{m,\tilde B_n}(\tilde W)|<\eta.    
\end{align}  
For $n\geq N$ and $0\leq j\leq 4(M+1)$ we set
\begin{align*}
    B_{n,j}^4:=B(c, e^{-4\alpha n+j2\alpha n/(4(M+1))}).
\end{align*}
Then, by the same argument as that used to obtain \eqref{eq finiteness}, for each $n\geq N$ there exists $0\leq q'(n)\leq 4(M+1)-1$ such that 
\begin{align}\label{eq decomposition no critical value claim 2}
\overline{B_{n,q'(n)+1}^4\setminus B_{n,q'(n)}^4} \cap \bigcup_{k=1}^{n} \{T^k(c)\}=\emptyset,    
\end{align}
where $\overline{A}$ denotes the Euclidean closure of $A\subset [0,1]$.
For each $n\geq N$, $W^4\in \mathcal{C}_n^4$, $0\leq m\leq n$ and $0\leq j\leq p(\text{Inc}_{B_n^4,\tilde B_n}(W^4))+1$ we set
\begin{align*}
    &B_n:=B_{n,q'(n)}^4,\ B_n':=B_{n,q'(n)+1}^4,\ 
    W:=\text{Inc}_{B_n^4,B_n}(W^4),
    W':=\text{Inc}_{B_n^4,B_n'}(W^4),\\&  
    \tilde W:=\text{Inc}_{B_n^4,\tilde B_n}(W^4),\ 
    V_{m}(W):=V_{m,B_n}(W)
    ,\ 
    V_{m}(W'):=V_{m,B_n'}(W')
  ,\ 
    k_{j}:=k_{j}^{\tilde W}
    .
\end{align*}
Note that, since $W^4\subset W\subset W'\subset \text{Inc}_3(W^4)$, $W=V_{n}(W)$, $W'=V_{n}(W')$ and $k_{p(\tilde W)+1}=n+1$,
for each $n\geq N$ and $W^4\in \mathcal{C}_n^4$, we have
\begin{align}\label{eq claim 2 small change}
    \frac{|W^4|}{|\text{Inc}_3(W^4)|}\leq \frac{|W|}{|W'|}
    =\frac{|V_{k_{p(\tilde W)+1}-1}(W)|}{|V_{k_{p(\tilde W)+1}-1}(W')|}.
\end{align}
Let $C_1$ be the constant obtained in Claim \ref{claim distorsion} and let $C_2$ be the constant obtained in Lemma \ref{lemma ratio}.
We fix $n\geq N_0$ and $W^4\in \mathcal{C}_n^4$.
Note that by \eqref{eq decomposition no critical value claim 2}, we have
\begin{align}\label{eq claim ratio boundary}
    T^n(\partial W) \subset \partial B_n \text{ and }T^n(\partial W') \subset \partial B_n'.
\end{align}
For each $0\leq m\leq n$ we set 
\[
 \hat V_m(W):=T^{n-m}(W)\subset V_m(W) \text{ and }\hat V_m(W'):=T^{n-m}(W')\subset V_m(W').
\]
By Claim \ref{claim distorsion} and the mean value theorem, for all $1\leq j\leq p(\tilde W)+1$ we have 
\begin{align}\label{eq claim 2 ratio diffeo}
    \frac{|\hat V_{k_j-1}(W)|}{|\hat V_{k_j-1}(W')|}\leq  C_1^2\frac{|\hat V_{k_{j-1}}(W)|}{|\hat V_{k_{j-1}}( W')|}.
\end{align}
Let $1\leq j\leq p(\tilde W)$.
Recall that by the construction, we have $c\in V_{k_j,\tilde B_n}(\tilde W)$. Moreover, by \eqref{eq decomposition no critical value claim 2}, if $c\in \hat V_{k_j}(W')$ then $c\in \hat V_{k_j}(W)$.
We first consider the case where $c\in \hat V_{k_j}(W)$. By (D1) of Lemma \ref{lemma ratio}, we obtain
\begin{align}\label{eq claim 2 ratio 1}
\frac{|\hat V_{k_j}(W)|}{|\hat V_{k_j}(W')|}\leq C_2\left(\frac{|\hat V_{k_j-1}(W)|}{|\hat V_{k_j-1}(W')|}\right)^{1/l}.    
\end{align}
Next, we consider the case where $c\notin \hat V_{k_j}(W)$. By \eqref{eq small preimage claim 2} and (D2) of Lemma \ref{lemma ratio}, we obtain
\begin{align}\label{eq claim 2 ratio 2}
\frac{|\hat V_{k_j}(W)|}{|\hat V_{k_j}(W')|}\leq C_2\left(\frac{|\hat V_{k_j-1}(W)|}{|\hat V_{k_j-1}(W')|}\right)^{1/l}.
\end{align}
By \eqref{eq claim 2 small change}, \eqref{eq claim 2 ratio diffeo}, \eqref{eq claim 2 ratio 1} and \eqref{eq claim 2 ratio 2}, we obtain
\begin{align}\label{eq claim 2 final}
    \frac{|W^4|}{|\text{Inc}_3(W^4)|} \leq C_1^{2(M+1)}C_2^M\left(\frac{|T^n(W)|}{|T^n(W')|}\right)^{l^{-M}}.
\end{align}

Moreover, either $T^n(W)=B_n$ and $T^n(W')=B_n'$ or there exists $x\in B_n \cap \bigcup_{k=1}^{n}\{T^k(c)\}$ such that
either
\begin{align*}
&
 \{x,c+e^{-4\alpha n+\frac{q'(n)2\alpha n}{4(M+1)}}\}=\partial T^n(W) \text{ and } 
 \{x,c+e^{-4\alpha n+\frac{(q'(n)+1)2\alpha n}{4(M+1)}}\}=\partial T^n(W') \text{ or }
\\&
\{x,c-e^{-4\alpha n+\frac{q'(n)2\alpha n}{4(M+1)}}\}=\partial T^n(W) \text{ and } 
 \{x,c-e^{-4\alpha n+\frac{(q'(n)+1)2\alpha n}{4(M+1)}}\}=\partial T^n(W').
\end{align*}
Let $N_1\ge N_0$ be a sufficiently large natural number such that $e^{-\alpha n/(2(M+1))}\le 2^{-1}$ for all $n\ge N_1$.
Then in all cases and for all $n\ge N_1$, we have
\[
\frac{|T^n(W)|}{|T^n(W')|}\leq 4e^{-\frac{\alpha n}{2(M+1)}}.
\]
Combining this with \eqref{eq claim 2 final}, we obtain the desired result.
\qed

We are now ready to complete the proof. Let $N_1$ be the natural number obtained in Claim \ref{claim exponentially small} and let $C_3>0$ be the constant obtained in Claim \ref{claim exponentially small}.  
Note that $W^3\cap \tilde W^3=\emptyset$ for all $W^3,\tilde W^3\in\mathcal{C}_n^3$ with $W^3\neq \tilde W^3$.
Thus, by \eqref{eq finite to 1},
we obtain
\[
\sum_{W^4\in \mathcal{C}_n^4}|\text{Inc}_3(W^4)|\leq 2^M\sum_{W^3\in \mathcal{C}_n^3}|W^3|\leq 2^M|[0,1]|=2^M.
\]
Combining this with Claim \ref{claim exponentially small}, we obtain
\begin{align}\label{eq slow main}
    &\sum_{n=1}^\infty  \Leb\left(\bigcup_{W^4\in \mathcal{C}_n^4}W^4\right)
    =\sum_{n=1}^\infty\sum_{W^4\in \mathcal{C}_n^4}|W^4|
    \\&\leq\sum_{n=1}^{N_1-1}\sum_{W^4\in \mathcal{C}_n^4}|W^4|
    +C_3\sum_{n=N_1}^\infty \sum_{W^4\in \mathcal{C}_n^4}|\text{Inc}_3(W^4)|e^{-\frac{\alpha n}{2(M+1)l^M}}\nonumber
    \\&\leq \sum_{n=1}^{N_1-1}\sum_{W^4\in \mathcal{C}_n^4}|W^4|
    +C_32^M \sum_{n=N_1}^\infty
    e^{-\frac{\alpha n}{2(M+1)l^M}}<\infty\nonumber
    .
\end{align}
For $k\in \mathbb{N}$ we set $\alpha_k=1/k$. By \eqref{eq slow main} and the Borel–Cantelli lemma, for each $k\in\mathbb{N}$ there exists a Borel set $A_k\subset [0,1]$ such that $\Leb(A_k)=1$ and for all $x\in A_k$ we have 
$
x\in \bigcup_{n\in\mathbb{N}} \bigcap_{j=n}^\infty \left(\bigcup_{W^4\in \mathcal{C}_j^4}W^4\right)^c. 
$
Hence, by setting $A:=\bigcap _{k\in \mathbb{N}} A_k$, we are done. 
\end{proof}

\section{Known results on $S$-unimodal maps}\label{sec known results}
In this section, we collect known results on $S$-unimodal maps that will be used to prove Theorem \ref{thm main}.

We begin with the following lemma due to Nowicki and Sands \cite{NowickiSands1998}.

\begin{lemma}{\cite[Lemma 4.8]{NowickiSands1998}}\label{lemma Nowicki and Sands}
    Let $T$ be an $S$-unimodal map with a non-flat critical point. Assume that $T$ has no periodic attractors. Then for every $0\leq\rho<1$ there exists a constant $K>0$ such that for each $x\in [0, T(c)]$ and $n\in\mathbb{N}$ satisfying $T^i(x)\leq x$ for all $0\leq i\leq n$, we have 
    \begin{align}\label{eq nowickisands}
        |(T^n)'(x)|\geq K\rho^n.
    \end{align}
\end{lemma}

Let $T$ be an $S$-unimodal map with a non-flat critical point. 
For $x\in [0,1]\setminus\{c\}$ there exists a unique point $\hat x\in [0,1]$ such that $x\neq \hat x$ and $T(x)=T(\hat x)$. We set $\hat c:=c$.
For each $n\in\mathbb{N}$ we define
\begin{align*}
    \mathcal{H}^n:=\{x\in (0,1): T^i(x)\notin(x,\hat x) \text{ for all $1\leq i\leq n-1$ and }T^n(x)\in (x,\hat x) \}
    .
\end{align*}
We also define 
\[
\lambda_\pi(T):=\inf\left\{
\frac{1}{n}\log |(T^n)'(y)|:
y\in[0,1],\ n\in\mathbb{N},\ T^n(y)=y
\right\}.
\]
Let 
\[
M_T:=\sup\left\{\frac{|T'(x)|}{|T'(\hat x)|}: x\in (0,1)\setminus\{c\}\right\}.
\]
\begin{rem}
    Let $T$ be an $S$-unimodal map. We assume that the order of the critical point $c$ of $T$ is $l$. Then we have
\begin{align}\label{eq finite symetric}
    M_T<\infty.
\end{align}
Indeed, by \eqref{eq critical order}, for any $x\in [0,1]\setminus\{c\}$ we have 
\begin{align}\label{eq elementary}
    &|T'(x)|
    \geq 
    \frac{|x-c|^{l}}{W|x-c|}
    \geq
    \frac{|T(x)-T(c)|}{W^2|x-c|}
    \\&=\frac{|T(\hat x)-T(c)|}{W^2|x-c|}
    \geq \frac{|\hat x-c|^l}{W^3|x-c|}
    \geq\frac{|T'(\hat x)|}{W^{2/l}W^4}, \nonumber
\end{align}
where $W$ denotes the constant in \eqref{eq critical order}.
\end{rem}

If $T$ has no periodic attractors, then by
\cite[(3.17)]{Keller1990}, we have 
\begin{align}\label{eq keller}
    \log |(T^n)'(x)|\geq -\log M_T+n\lya_\pi(T) \text{ for all $n\in\mathbb{N}$ and $x\in\mathcal{H}^n$}.
\end{align}
We set
\begin{align*}
&\mathcal{R}:=\left\{x\in [0,1]\setminus
\bigcup_{n=0}^\infty T^{-n}(c)
:
\liminf_{n\to\infty}|T^{n}(x)-c|=0\right\}.
\end{align*}
Let $x\in \mathcal{R}$. Then the following construction of the sequence $\{n_k(x)\}_{k\in \mathbb{N}\cup\{0\}}$ is well defined:
\begin{itemize}
    \item[(S0)] We set $n_0(x):=0$.
    \item[(S1)] We define $n_1(x)$ to be the smallest integer satisfying $T^{n_1(x)}(x)\in(x,\hat{x})$. 
    \item[(S2)] Suppose that $n_k(x)$ has already been defined for some $k\in\mathbb{N}$. We define $n_{k+1}(x)$ to be the smallest integer satisfying 
    \[T^{n_{k+1}(x)}(x)\in (T^{n_k(x)}(x), \hat{T^{n_k(x)}(x)}).\] 
\end{itemize}
Note that, by the construction, for all $k\in \mathbb{N}\cup\{0\}$ we have 
\begin{align}\label{eq H decomposition}
    T^{n_{k}(x)}(x)\in \mathcal{H}^{n_{k+1}(x)-n_{k}(x)}.
\end{align}

\begin{rem}
Let $T$ be an $S$-unimodal map.
The assumption that $T''(c)\neq 0$ in Keller \cite[Theorem 3 (b)]{Keller1990} is only used to prove \cite[Lemma 4]{Keller1990}. 
Moreover, within the proof of \cite[Lemma 4]{Keller1990}, this assumption is only used to prove that $M_T<\infty$. 
Thus, by \eqref{eq finite symetric}, we can replace the assumption that $T''(c)\neq 0$ by the weaker assumption that the critical point of $T$ is non-flat.
\end{rem}

Keller showed the following lemma in the proof of \cite[Lemma 4]{Keller1990} (see \cite[(3.19)]{Keller1990}):
\begin{lemma}\label{lemma keller}
    Let $T$ be an $S$-unimodal map with a non-flat critical point. Assume that $T$ has no periodic attractors. Then there exists a Borel measurable set $G(T)\subset\mathcal{R}$ such that $\Leb(G(T))=1$ and for all $x\in G(T)$ we have
    \begin{align}\label{eq stap big}
        \lim_{k\to\infty}(n_{k+1}(x)-n_k(x))=\infty.
    \end{align}
\end{lemma}

Keller \cite[Theorem 3 (b)]{Keller1990} also showed the following theorem (see also \cite[Chapter V, Theorem 3.2]{de2012one}):

\begin{thm}\label{thm keller}{\cite{Keller1990}}
    Let $T$ be an $S$-unimodal map with a non-flat critical point. Then there exists a constant $\ulya_T\in\mathbb{R}\cup\{-\infty\}$ such that for $\Leb$-almost every point $x\in [0,1]$ we have 
    $\ulya_T=\ulya(x)$.
    Moreover, $T$ has an absolutely continuous $T$-invariant probability measure with  positive entropy if and only if $\ulya_T>0$.
    Furthermore, $T$ has a strictly stable periodic orbit if and only if $\ulya_T<0$.
\end{thm}

The following theorem was proved by Przytycki \cite{characteristic1993}. It was later generalized by Rivera-Letelier \cite{Rivera-Letelier} to a more general setting.

\begin{thm}{\cite[Theorem B]{characteristic1993}, \cite[Proposition A.1]{Rivera-Letelier}}\label{thm Przytycki}
    Let $T$ be an $S$-unimodal map with a non-flat critical point and let $\mu$ be a $T$-invariant Borel probability measure. We assume that $T$ has no periodic attractors. Then we have
    \[
    \int \log |T'| d\mu\geq 0.
    \]
\end{thm}

We end this section with the following well-known theorem (see, for example, \cite[Chapter V, Theorem 1.2]{de2012one}).
\begin{thm}\label{thm ergodicity}
    Let $T$ be an $S$-unimodal map with a non-flat critical point. We assume that $T$ has no periodic attractor. Then the Lebesgue measure $\Leb$ on $[0,1]$ is ergodic with respect to $T$, that is, for every Borel measurable set $A\subset [0,1]$ satisfying $T^{-1}(A)=A$ we have $\Leb(A)\in \{0,1\}$.
\end{thm}

\section{Proof of Theorem \ref{thm main}}\label{sec proof of theorem main}

In this section, we give a proof of Theorem \ref{thm main}.
The proof of Theorem \ref{thm main} is based on the following result of independent interest.
\begin{prop}\label{prop nonegative}
    Let $T$ be an $S$-unimodal map with a non-flat critical point. Assume that $T$ has no periodic attractors. Then for $\Leb$-almost every $x\in [0,1]$ we have 
    $\llya(x)\geq 0$.
\end{prop}

\begin{proof}
Let $l\in (1,\infty)$ be the order of the critical point $c$. By Proposition \ref{prop slowly recurrent} and Lemma \ref{lemma keller}, there exists $A\subset G(T)$ such that $\Leb(A)=1$ and for all $x\in A$ the orbit $O(x)$ of $x$ satisfies the slowly recurrent condition. Here, $G(T)$ denotes the set obtained in Lemma \ref{lemma keller}.
Let $x\in A$.
Let $\alpha>0$ and let $0<\rho<1$. Then there exist constants $C>0$ and $N\in \mathbb{N}$ such that \eqref{eq slowly recurrent} holds. For each $j\in\mathbb{N}$ there exists a unique integer $k(j)$ such that 
\begin{align}\label{eq def j}
n_{k(j)}(x)\leq j<n_{k(j)+1}(x).
\end{align}
We choose $\tilde N\in\mathbb{N}$ sufficiently large so that for all $j\geq \tilde N$ we have $n_{k(j)}(x)\geq N$.
Let $j\geq \tilde N$.
We set 
\begin{align}\label{eq def l(j)}
\ell(j):=j-n_{k(j)}(x)\in \{0,1,\cdots,n_{k(j)+1}(x)-n_{k(j)}(x)-1\}.
\end{align}
Since we have
\begin{align*}
    &j=\ell(j)-1+1+n_{k(j)}(x)
    =\ell(j)-1+1+\sum_{m=1}^{k(j)}(n_{m}(x)-n_{m-1}(x)),
\end{align*}
if $\ell(j)\geq 1$ then we obtain
\begin{align}\label{eq decomposition}
    &\log |(T^j)'(x)|= 
    \log \left|\left(T^{\ell(j)-1}\right)'\left(T^{1+n_{k(j)}(x)}(x)\right)\right|
    \\&+\log \left|T'\left(T^{n_{k(j)}(x)}(x)\right)\right|
    +\sum_{m=1}^{k{(j)}}\log\left|\left(T^{n_{m}(x)-n_{m-1}(x)}\right)'\left(T^{n_{m-1}(x)}(x)\right)\right|,\nonumber
\end{align}
where $T^0$ is defined to be the identity map.
If $\ell(j)=0$ then we have
\begin{align}\label{eq decomposition 2}
    \log |(T^j)'(x)|=\sum_{m=1}^{k{(j)}}\log\left|\left(T^{n_{m}(x)-n_{m-1}(x)}\right)'\left(T^{n_{m-1}(x)}(x)\right)\right|.
\end{align}
If $\ell(j)>1$ then by the construction of the sequence $\{n_{k}(x)\}_{k\in\mathbb{N}}$ and \eqref{eq def j}, for all $0<i\leq \ell(j)-1$ we have 
\[T^i(T^{1+n_{k(j)}(x)}(x))=T^{1+n_{k(j)}(x)+i}(x)\leq T^{1+n_{k(j)}(x)}(x).\]
Hence, by Lemma \ref{lemma Nowicki and Sands}, we obtain
\begin{align}\label{eq nonnegative 1}
    \log \left|\left(T^{\ell(j)-1}\right)'\left(T^{1+n_{k(j)}(x)}(x)\right)\right|\geq \log K + (\ell(j)-1)\log \rho,
\end{align}
where $K$ denotes the constant obtained in Lemma \ref{lemma Nowicki and Sands}. Without loss of generality, we may assume that $K\leq 1$.
Then \eqref{eq nonnegative 1} also holds
when $\ell(j)=1$.
By \eqref{eq critical order} and \eqref{eq slowly recurrent}, we have 
\begin{align}\label{eq nonnegative 2}
    &\log \left|T'\left(T^{n_{k(j)}(x)}(x)\right)\right|
    \geq 
    -\log W+(l-1) \log |T^{n_{k(j)}(x)}(x)-c|
    \\&\geq -\log W+(l-1)\log C-\alpha(l-1) n_{k(j)}(x), \nonumber
\end{align}
where $W$ denotes the constant in \eqref{eq critical order}. 
By \eqref{eq keller} and \eqref{eq H decomposition}, we have
\begin{align}\label{eq nonnegative 3}
    &\sum_{m=1}^{k(j)}\log\left|\left(T^{n_{m}(x)-n_{m-1}(x)}\right)'\left(T^{n_{m-1}(x)}(x)\right)\right|
    \\&
    \geq
    \sum_{m=1}^{k(j)}
    \left(
    -\log M_T
    +(n_{m}(x)-n_{m-1}(x))\lya_\pi(T)
    \right) 
    \geq -k(j)\log M_T +n_{k(j)}(x) \lya_\pi(T). \nonumber
\end{align}
By \eqref{eq decomposition}, \eqref{eq decomposition 2}, \eqref{eq nonnegative 1}, \eqref{eq nonnegative 2} and \eqref{eq nonnegative 3}, we obtain
\begin{align}\label{eq nonegative midle}
    \log |(T^j)'(x)|
    \geq 
    &\min\{\log KW^{-1}C^{l-1},0\} -k(j)\log M_T\\&+\max\{\ell(j)-1,0\}\log \rho
     - n_{k(j)}(x)\alpha(l-1)+ n_{k(j)}(x)\lya_\pi(T).\nonumber
\end{align}
We will show that 
\begin{align}\label{eq nonegative 4}
    \lim_{j\to\infty}\frac{k(j)}{j}=0.
\end{align}
Let $Q>0$. 
By \eqref{eq stap big}, there exists $\tilde L\in\mathbb{N}$ such that for all $k\geq \tilde L$ we have 
\begin{align}
    n_k(x)-n_{k-1}(x)\geq Q.
\end{align}
Note that, since $\{n_k(x)\}_{k\in\mathbb{N}}$ is monotone increasing, we have $\lim_{j\to\infty} k(j)=\infty$.
For a sufficiently large $j\in\mathbb{N}$ with $k(j)\geq \tilde L$ we have 
\begin{align*}
    \frac{k(j)}{j}\leq \frac{k(j)}{\sum_{m=\tilde L}^{k(j)}(n_{m}(x)-n_{m-1}(x))}
    \leq \frac{k(j)}{(k(j)-\tilde L+1)Q}.
\end{align*}
Hence, $\limsup_{j\to \infty}k(j)/j\leq Q^{-1}$. Letting $Q\to\infty$, we obtain \eqref{eq nonegative 4}. 

By \eqref{eq def j} and \eqref{eq def l(j)}, we have $\max\{\ell(j)-1,0\}\leq j$ and $n_{k(j)}(x)\leq j$. Therefore, since $-\alpha(l-1)<0$ and $\log \rho<0$ we have 
\begin{align}\label{eq nonegative 5}
    \frac{\max\{\ell(j)-1,0\}}{j}\log \rho\geq \log \rho \text{ and }
    - \frac{n_{k(j)}(x)}{j}\alpha(l-1)\geq -\alpha(l-1).
\end{align}
Since $T$ has no periodic attractor, we have $\lya_\pi(T)\geq 0$. Therefore, by using \eqref{eq nonegative midle}, \eqref{eq nonegative 4} and \eqref{eq nonegative 5} and noting that the constants $K,W,C, M_T,l$ do not depend on $j$, we obtain
$
 \llya(x)\geq \log \rho -\alpha(l-1).
$
Letting $\rho\to 1$ and $\alpha\to 0$, we obtain the desired result.
\end{proof}

\emph{Proof of Theorem \ref{thm main}.}
We denote by $\ulya_T$ the number obtained in Theorem \ref{thm keller} and set 
\[
\lya_T:=\ulya_T.
\]
   By \cite[Chapter V, Theorem 1.3]{de2012one}, if $T$ has a periodic attractor then there exists a periodic point $p$ with period $q\in\mathbb{N}$ such that for $\Leb$-almost every $x\in [0,1]$ we have $\omega(x)=O(p)$. By Theorem \ref{thm keller} this implies that for $\Leb$-almost every $x\in [0,1]$ we have 
   \[
   \lim_{n\to\infty}\frac{1}{n}\log |(T^n)'(x)|=\frac{1}{q}\log |(T^q)'(p)|=\lya_T.
   \]

We assume that $T$ has no periodic attractor.
If $T$ has an absolutely continuous  $T$-invariant probability measure $\mu_T$ with positive entropy then by Theorem \ref{thm ergodicity}, the measure $\mu_T$ is ergodic with respect to $T$. Furthermore, by Theorem \ref{thm Przytycki}, $\log |T'|$ is $\mu_T$-integrable. Therefore, by Birkhoff's ergodic theorem, 
the set 
\[
L_T:=\left\{x\in [0,1]\setminus\bigcup_{n=0}^\infty T^{-n}(\{c\}):\lim_{n\to\infty}\frac{1}{n}\log |(T^n)'(x)|=\int \log |T'|d\mu_T\right\}
\]
has full measure with respect to $\mu_T$. In particular, it has positive Lebesgue measure. Therefore, since we have $T^{-1}(L_T)\setminus L_T\cup L_T\setminus T^{-1}(L_T)\subset \{c\}$, Theorem \ref{thm ergodicity} implies that $\Leb(L_T)=1$.

We assume that $T$ has no absolutely continuous $T$-invariant probability measure with positive entropy. Then, by Theorem \ref{thm keller}, for $\Leb$-almost every point $x\in [0,1]$ we have 
   $ \ulya(x)=\lya_T\leq 0.$
    Combining this with Proposition \ref{prop nonegative}, we conclude that 
    for $\Leb$-almost every point $x\in [0,1]$ we have $\lya(x)=\lya_T=0$.
    The last statement of Theorem \ref{thm main} follows from Theorem \ref{thm keller}.
\qed

\section{Proof of Theorem \ref{thm quadratic map}}\label{sec proof of theorem quadratic}
This section is devoted to the proof of Theorem \ref{thm quadratic map}.
For $a\in [0,4]$ we define the quadratic map $\map:[0,1]\rightarrow [0,1]$ by 
\[
\map(x)=ax(1-x).
\]
Let $a\in [0,4]$. We assume that the critical point $c=1/2$ of $\map$ is not a periodic point. 
We first recall a theorem of Hofbauer and Keller \cite[Theorem 4]{HofbauerKeller}. To do this, we describe kneading sequences and kneading maps.  
We define the kneading sequence $\kneading=e_1e_2\cdots\in \{0,1\}^{\mathbb{N}}$ of $\map$ as follows: For each $k\in\mathbb{N}$ we set $e_k:=0$ if $\map^k(c)\in [0,c)$ and $e_k:=1$ if $\map^k(c)\in (c,1]$.

A map $Q:\mathbb{N}\rightarrow\mathbb{N}\cup\{0\}$ is called a kneading map if the following conditions hold:
\begin{itemize}
    \item[(Q1)] For all $k\in\mathbb{N}$ we have $Q(k)<k$.
    \item[(Q2)] For all $k\in\mathbb{N}$ with $Q(k)\geq 1$ we have 
    \begin{align*}
        \{Q(j+k)\}_{j\in\mathbb{N}}\geqq \{Q(Q(Q(k))+j)\}_{j\in\mathbb{N}},
    \end{align*}
    where $\geqq$ denotes the lexicographic ordering.
\end{itemize}
Let $Q$ be a kneading map. We construct $\{r(k)\}_{k\in\mathbb{N}\cup\{0\}}$ and $\{S(k)\}_{k\in\mathbb{N}\cup\{0\}}$ as follows.
We set $r(0):=1$ and $S(0):=1$.
We also define $r(1):=S(Q(1))=S(0)=1$ and $S(1):=r(0)+r(1)$. Assume that, for some $k\in \mathbb{N}$, numbers $S(j)$ and $r(j)$ have already been defined for all $1\leq j\leq k$. We set 
\begin{align*}
    r(k+1):=S(Q(k+1)) \text{ and }S(k+1):=\sum_{j=0}^{k+1} r(j).
\end{align*} 
By the definition of $\{r(k)\}_{k\in\mathbb{N}\cup\{0\}}$ and $\{S(k)\}_{k\in\mathbb{N}\cup\{0\}}$ and (Q1), for all $k\in\mathbb{N}$ we have 
\begin{align}\label{eq upper bound S}
    S(k)\leq 2^k.
\end{align}
Using these sequences, we define the sequence $\kneading=e_1e_2\cdots\in \{0,1\}^\mathbb{N}$ as follows. We set $e_1=e_{S(0)}:=1$ and $e_2=e_{S(1)}=0$. Assume that, for some $k\in \mathbb{N}$, the symbols $e_j$ have already been defined for all $1\leq j\leq S(k)$. We set 
\begin{align*}
    e_{S(k)+j}=e_j \text{ for all }1\leq j< r(k+1)
    \text{ and }
    e_{S(k+1)}\neq e_{r(k+1)}.
\end{align*}
We refer to the sequence $\kneading$ as the $Q$-sequence of the kneading map $Q$. 
\begin{thm}\label{thm Hofbauer Keller}
    Let $Q$ be a kneading map. Then there exists $a\in [0,4]$ such that the critical point of $\map$ is not periodic and the kneading sequence of the quadratic map $\map$ is the $Q$-sequence of the kneading map $Q$.
\end{thm}
Following Hofbauer and Keller \cite{HofbauerKeller}, we also introduce the following notion. A pair 
$\mathcal{F}:=
\{\{V_k\}_{k\in \mathbb{N}\cup\{0\}},
\{U_k\}_{k\in \mathbb{N}\cup\{0\}}\}$
of sequences of integers is a frame if $V_0=0$, $V_{k-1}<U_k<V_k$ for all $k\in\mathbb{N}$, 
\begin{align}\label{eq frame U}
    U_{k+1}\geq k2^{k+V_k} \text{ for all }k\in\mathbb{N}
\end{align}
and 
\begin{align}\label{eq frame V}
    V_k\geq k^22^{U_k} \text{ for all }k\in\mathbb{N}.
\end{align}
For a frame $\mathcal{F}$ we define the skeleton $\mathscr{S}(\mathcal{F})$ as the set of kneading maps satisfying 
\begin{align}\label{eq skeleton 1}
Q(i)=U_k
   \text{ for all } k\in\mathbb{N} \text{ and } U_k<i\leq V_k
\end{align}
and 
\begin{align}\label{eq skeleton 2}
    Q(U_{k+1})<U_k \text{ for all }k\in\mathbb{N}.
\end{align}
The following proposition was also proved by Hofbauer and Keller \cite[Proposition 1]{HofbauerKeller}:
\begin{prop}\label{prop Hofbauer Keller}{\cite[Proposition 1]{HofbauerKeller}}
    Let $N\in\mathbb{N}$. Then there are uncountably many different frames $\mathcal{F}$ with $U_1=N+1$ such that for each $Q\in \mathscr{S}(\mathcal{F})$ and the parameter $a$ obtained in Theorem \ref{thm Hofbauer Keller} we have the following: (1) $\map$ has no ergodic absolutely continuous invariant probability measure with positive entropy, and (2) for $\Leb$-almost every $x\in [0,1]$ and every continuous function $\psi$ on $[0,1]$ we have 
    \begin{align*}
        \lim_{n\to\infty}\left(\frac{1}{S({V_n})}\sum_{k=0}^{S(V_n)-1}\psi(\map^{k}(x))
        -
        \frac{1}{S({U_n})}\sum_{k=0}^{S(U_n)-1}
        \psi\left(\map^{k}\left(f_a(c)\right)\right)
        \right)=0.
    \end{align*}
\end{prop}

We  endow $\{0,1\}^\mathbb{N}$ with the shift metric.
We denote by $\sigma$ the left shift map on $\{0,1\}^\mathbb{N}$. 
For $N\in\mathbb{N}$, as in \cite{HofbauerKeller}, we define 
\begin{align*}
\Omega_N:=\left\{\omega\in\{0,1\}^{\mathbb{N}}
:\begin{array}{l} 0^N 
\text{ and }01^{2i+1}0 \text{ for }i\in\mathbb{N}\cup\{0\} \text{ do not occur}\\
\text{as subwords of }\omega
\end{array}
\right\}.
\end{align*}
Note that $\Omega_N$ is closed and $\sigma(\Omega_N)=\Omega_N$.
We denote by $M(\Omega_N)$ the space of Borel probability measures on $\Omega_N$ endowed with the weak* topology.
Let $\tilde d$ be a metric on $M(\Omega_N)$ generating the weak* topology.
We denote by $M_\sigma(\Omega_N)$ the set of all $\sigma$-invariant 
Borel probability measures on $\Omega_N$.

We are now in a position to begin the proof of Theorem \ref{thm quadratic map}. 
The construction of the kneading map $Q$ in the following proof is exactly the same as that in the proof of \cite[Proposition 2]{HofbauerKeller}.

\emph{Proof of Theorem \ref{thm quadratic map}.}
Let $N\geq 2$. We fix a frame $\mathcal{F}$ obtained in Proposition \ref{prop Hofbauer Keller} with $U_1=N+1$ (i.e. $\mathcal{F}$ satisfies conditions (1) and (2) of Proposition \ref{prop Hofbauer Keller}). Let $\tilde p_1=111\cdots$ and let $\tilde p_3:=011011011\cdots$. Then we have $\tilde p_1\in \Omega_N$ and $\tilde p_3\in \Omega_N$. Let 
\[
C:=\left\{t\delta_{\tilde p_1}+(1-t)\left(\frac{1}{3}\sum_{i=0}^2\delta_{\sigma^i(\tilde p_3)}\right)
:t\in [0,1]
\right\}.
\]
Notice that $C$ is a closed convex subset of $M_\sigma(\Omega_N)$.
We write 
$
\mathbb{Q}\cap[0,1]=\{q_i:i\in\mathbb{N}\}
$, where $q_1=1$.
Let 
\[
C':=\left\{\mu_i:=q_i\delta_{\tilde p_1}+\left(1-q_i\right)\left(\frac{1}{3}\sum_{i=0}^2\delta_{\sigma^i(\tilde p_3)}\right)
:i\in \mathbb{N}
\right\}.
\]
Then $C'$ is a dense subset of $C$. By \cite[p. 324]{HofbauerKeller}, $\Omega_N$ has the specification property (see, for example, \cite[Definition 21.1]{DenkerGrillendergerSigmund} for the definition of the specification property). Thus, by \cite[Corollary 21.15]{DenkerGrillendergerSigmund}, for each $i\in\mathbb{N}$, $\mu_i$ has a generic point $\omega(i)\in \Omega_N$ satisfying $\omega_1(i)=0$, that is, for any continuous function $\psi$ on $\Omega_N$ we have $\lim_{n\to\infty}\frac{1}{n}\sum_{j=0}^{n-1}\psi(\sigma^j(\omega(i)))=\int \psi d\mu_i$. 
Since $\mu_1=q_1 \delta_{\tilde p_1}=\delta_{\tilde p_1}$, we can choose 
\begin{align}\label{eq def omega 1}
\omega(1)=011111\cdots.    
\end{align}
Hence, for each $i\in \mathbb{N}$ there exists $l(i)\in\mathbb{N}$ such that for all $l\geq l(i)$ we have 
\[
\tilde d\left(\mu_i,\frac{1}{l}\sum_{j=1}^l\delta_{\sigma^j(\omega(i))}\right)<\frac{1}{i}.
\]
We set $\tilde N:=l(1)+1$ and $n_k=1$ for all $1\leq k\leq \tilde N$. 
For all $m\in\mathbb{N}$ and  $\max_{1\leq j\leq m}\{l(j)\}+\sum_{j=1}^mj+1\leq k
\leq 
\max_{1\leq j\leq m+1}\{l(j)\}+\sum_{j=1}^{m+1}j$ we set 
\[
n_k=\min\left\{k-\left(\max_{1\leq j\leq m}\{l(j)\}+\sum_{j=1}^mj\right)
,m
\right\}.
\]
Then for each $i\in\mathbb{N}$ there exists an infinite set $A\subset \mathbb{N}$ such that for all $k\in A$ we have $n_k=i$. Moreover, for all $k\geq \tilde N+1$ we have 
\begin{align*}
    S(V_k)\geq k^2\geq k l(n_k) \text{ and }
    S(U_{k+1})-S(V_k+k)\geq (k-1)2^{V_k+k}\geq l(n_k)
\end{align*}
by \eqref{eq frame U}, \eqref{eq frame V} and \eqref{eq upper bound S}. 
Recall that for all $i\in\mathbb{N}$ we have $\omega(i)\in \Omega_N$ and $\omega_1(i)=0$. 
For $k\in \mathbb{N}$ we can write $\omega(n_k)=v_1(n_k)v_2(n_k)\cdots$, where $v_j(n_k)=0$ or $v_j(n_k)=11$ for all $j\in \mathbb{N}$. 
As in the proof of \cite[Proposition 2]{HofbauerKeller}, we define the map $Q:\mathbb{N}\rightarrow \mathbb{N}\cup\{0\}$ by 
\begin{itemize}
    \item[(D1)]$Q(j):=0 \text{ for all }j=1,\cdots,U_1$.
    \item[(D2)]$Q(U_k+j):=U_k \text{ for all }k\geq1,\ j=1,\cdots,V_k-U_k.$
    \item[(D3)]$Q(V_k+j):=U_{k-j} \text{ for all } k\geq1,\ j=1,\cdots, k-1$.
    \item[(D4)] $Q(V_k+k):=1 \text{ for all }k\geq1$.
    \item[(D5)]
    $\text{ For all } k\geq1 \text{ and } j=1,\cdots, U_{k+1}-V_k-k$
    \begin{align}\nonumber
     &Q(V_k+k+j):=\left\{
 \begin{array}{ll}\nonumber
   0   & \text{if}\ v_j(n_k)=0 \\
   1   & \text{if} \ v_j(n_k)=11
 \end{array}\nonumber
 \right..
   \nonumber
 \end{align}
\end{itemize}
By the definition of $Q$, one can show that   
$     Q\in \mathscr{S}(\mathcal{F})$
and
\begin{align}\label{eq main kneading}
    &e_{S(V_k+k)+j}=\omega_j(n_k)
    \text{ for all }k\in\mathbb{N} \text{ and } 
    1\leq j\leq S(U_{k+1})-S(V_k+k)
\end{align}
(see the proof of \cite[Proposition 2]{HofbauerKeller}).
Here, $\kneading=e_1e_2\cdots$ denotes the $Q$-sequence of the kneading map $Q$ and $\omega(n_k)=\omega_1(n_k)\omega_2(n_k)\cdots$.

Let $a$ be the parameter obtained in Theorem \ref{thm Hofbauer Keller}. Then the kneading sequence of the quadratic map $\map$ is the $Q$-sequence $\kneading$ of the kneading map $Q$. 
If $\map$ has a periodic attractor associated with a periodic point $p$, then the $\omega$-limit set $\omega(c)$ of the critical point $c$ of $\map$ is the orbit $O(p)$ of $p$ (see \cite[Theorem 2.7]{Singer} or \cite[Chapter II, Theorem 6.1]{de2012one}). In this case, the kneading sequence $\kneading$ is eventually periodic. However, by \eqref{eq main kneading}, $\kneading$ is not eventually periodic. Hence, $\map$ has no periodic attractor. Combining this with \cite[Theorem 2.6]{Gukenheimer}, we conclude that there is no non-trivial interval $J\subset [0,1]$ such that $\map^n|_J$ is a homeomorphism for all $n\in\mathbb{N}$. Hence, the following coding map is a homeomorphism: We define the coding map 
$\pi:
[0,1]\setminus\bigcup_{n=0}^\infty \map^{-n}(c)
\rightarrow 
\Sigma$ by 
\begin{align*}
\pi_{k}(x)=0 \text{ if }\map^{k-1}(x)\in[0,c) \text{ and }
\pi_{k}(x)=1 \text{ if }\map^{k-1}(x)\in(c,1],
\end{align*}
where $\pi(x):=\pi_1(x)\pi_2(x)\cdots$ and $\Sigma:=\pi([0,1]\setminus\bigcup_{n=0}^\infty \map^{-n}(c))\subset \{0,1\}^\mathbb{N}$. 
Note that we have  $\kneading=\pi(\map(c))$ and $\pi\circ \map=\sigma\circ \pi$.
Let $p_1$ be the fixed point of $\map$ in $(c,1]$. Since $e_1=1$, $e_2=0$ and $e_3=0$, there exists a periodic point $p_3$ such that $\map^3(p_3)=p_3$ and $\pi(p_3)=\tilde p_3$. Let $L_a$ be the set of weak accumulation points of the sequence 
\[
\left\{\frac{1}{S(U_k)}\sum_{m=0}^{S(U_k)-1}\delta_{\map^m(\map(c))}\right\}_{k\in\mathbb{N}}.
\]
By \cite[Proposition 2]{HofbauerKeller}, we have 
\[
\overline{\omega}_a(\delta_{\map(c)})=L_a=
\left\{t\delta_{ p_1}+(1-t)\left(\frac{1}{3}\sum_{i=0}^2\delta_{\map^i(p_3)}\right)
:t\in [0,1]
\right\}.
\]
Moreover, by using this equality and Proposition \ref{prop Hofbauer Keller}, Hofbauer and Keller \cite[Theorem 5]{HofbauerKeller} showed that for $\Leb$-almost every $x\in [0,1]$ we have 
\[
\overline{\omega}_a(\delta_x)=\overline{\omega}_a(\Leb)=\overline{\omega}_a(\delta_{\map(c)})=\left\{t\delta_{ p_1}+(1-t)\left(\frac{1}{3}\sum_{i=0}^2\delta_{\map^i(p_3)}\right)
:t\in [0,1]
\right\}.
\]
This implies (1) and (2) of Theorem \ref{thm quadratic map}.

Let $\{k_j\}_{j\in\mathbb{N}}\subset \mathbb{N}$ be a strictly increasing sequence such that $n_{k_j}=1$ for all $j\in\mathbb{N}$.
Note that since $\map$ has no periodic attractor, we have $\log |\map'(p_1)|>0$. 
We will show that 
\begin{align}\label{eq upper lyapunov critical}
    \lim_{j\to\infty}\frac{1}{S(U_{k_j+1})-2}
    \log \left|\left(\map^{S(U_{k_j+1})-2}\right)'(\map(c))\right|
    =\log |\map'(p_1)|>0.
\end{align}
Let $\epsilon>0$ and let $0<\rho<1$. Since the function $\tau\in \Sigma\mapsto \log |\map'\circ\pi^{-1}(\tau)|\in\mathbb{R}$ is continuous at $\pi(p_1)=\tilde p_1$, there exists $N_1\geq 4$ such that for all $\tau\in \Sigma$ with $\tau_i=1$ for all $1\leq i\leq N_1$ we have 
\begin{align}\label{eq continuous}
    \left|\log |\map'\circ\pi^{-1}(\tau)|-\log |\map'\circ\pi^{-1}(\tilde p_1)|\right|<\epsilon.
\end{align}
We fix $N_2\in\mathbb{N}$ such that for all $j\geq N_2$ we have $S(U_{k_j+1})\geq N_1+S(V_{k_j}+k_j)+2$. 
By the chain rule, for all $j\geq N_2$ and $x\in[0,1]\setminus \bigcup_{n=1}^\infty \map^{-n}(c)$ we have 
\begin{align}\label{eq decomposition lyapunov 1}
     &\log \left|\left(\map^{S(U_{k_j+1})-2}\right)'(\map(x))\right|
     =
     \log \left|\left(\map^{S(V_{k_j}+k_j)+1}\right)'(\map(x))\right|
     \\&+
     \sum_{k=S(V_{k_j}+k_j)+2}^{S(U_{k_j+1})-(N_1-1)}
     \log |\map'(\map^{k}(x))|
     +
     \sum_{k=S(U_{k_j+1})-(N_1-2)}^{S(U_{k_j+1})-2}
     \log |\map'(\map^{k}(x))|\nonumber.
\end{align}
Let $K$ be the constant obtained in Lemma \ref{lemma Nowicki and Sands}. We set $\tilde K:=\max\{|\log a|, |\log \rho|\}$. By Lemma \ref{lemma Nowicki and Sands}, we obtain
\begin{align*}
\left|     \log \left|\left(\map^{S(V_{k_j}+k_j)+1}\right)'(\map(c))\right|\right|\leq |
     \log K|+\tilde K(S(V_{k_j}+k_j)+1).
\end{align*}
Moreover, by \eqref{eq upper bound S}, \eqref{eq frame U} and \eqref{eq frame V}, we have 
\begin{align*}
    \limsup_{j\to\infty}\frac{S(V_{k_j}+k_j)+1}{S(U_{k_j+1})-2}
    \leq
    \limsup_{j\to\infty}\frac{2^{V_{k_j}+k_j}+1}{k_j2^{V_{k_j}+k_j}-2}=0.
\end{align*}
By using these estimates, for each $x\in \{c,p_1\}$ we obtain 
\begin{align}\label{eq first term}
    \lim_{j\to \infty}\frac{1}{S(U_{k_j+1})-2}\log \left|\left(\map^{S(V_{k_j}+k_j)+1}\right)'(\map(x))\right|=0.
\end{align}
Recall that $\omega(n_{k_j})=\omega(1)=011111\cdots$ for all $j\in\mathbb{N}$ (see \eqref{eq def omega 1}).
By \eqref{eq main kneading}, for all $S(V_{k_j}+k_j)+2\leq k\leq S(U_{k_j+1})-(N_1-1)$ and $0\leq i\leq N_1-1$ we have $e_{k+i}=1$. Thus, by \eqref{eq continuous}, we obtain
\begin{align}\label{eq second term}
     & \frac{1}{S(U_{k_j+1})-2} \left|  \sum_{k=S(V_{k_j}+k_j)+2}^{S(U_{k_j+1})-(N_1-1)}\left(
     \log |\map'(\map^{k}(c))|-    
     \log |\map'(\map^{k}(p_1))|\right)
     \right|\leq
     \\&   \frac{1}{S(U_{k_j+1})-2} \sum_{k=S(V_{k_j}+k_j)+2}^{S(U_{k_j+1})-(N_1-1)}
     \left|\log |\map'\circ\pi^{-1}(\sigma^{k-1}(\kneading))|- 
     \log |\map'\circ\pi^{-1}(\tilde p_1)|
     \right|<\epsilon\nonumber.
\end{align}
By \eqref{eq main kneading}, for all $S(U_{k_j+1})-(N_1-2)\leq k\leq S(U_{k_j+1})-2$ we have $e_k=e_{k+1}=e_{k+2}=1$. Since $e_1=1$ and $e_2=0$ and $\pi^{-1}$ is continuous, there exists $\delta>0$ such that for all $j\in \mathbb{N}$ and $S(U_{k_j+1})-(N_1-2)\leq k\leq S(U_{k_j+1})-2$ we have $|\map^k(c)-c|>\delta$ and thus, 
\[
C:=\sup_{j\in\mathbb{N}}\max\left\{\left|\log |\map'(\map^k(c))|\right|:S(U_{k_j+1})-(N_1-2)\leq k\leq S(U_{k_j+1})-2\right\}<\infty.
\]
This implies that for $x\in \{c,p_1\}$ we obtain
\begin{align*}
\lim_{j\to\infty}    \frac{1}{S(U_{k_j+1})-2}     \sum_{k=S(U_{k_j+1})-(N_1-2)}^{S(U_{k_j+1})-2}
     \log |\map'(\map^{k}(x))|=0.
\end{align*}
    Combining this with \eqref{eq decomposition lyapunov 1}, \eqref{eq first term} and \eqref{eq second term}, we obtain \eqref{eq upper lyapunov critical}.

Note that since $[0,1]$ is compact, for each continuous function $\phi$ on $[0,1]$ we have $\sup_{x\in [0,1]}\{|\phi(x)|\}<\infty$. By an argument similar to that used to prove \eqref{eq upper lyapunov critical}, we obtain
\begin{align}\label{eq delta limit}
\lim_{j\to\infty}
\frac{1}{S(U_{k_j+1})-2}\sum_{m=1}^{S(U_{k_j+1})-2}\delta_{\map^m(c)}=\delta_{p_1}.
\end{align}
It is known that if we have $\llya(f_{\tilde a}(c))>0$ for some parameter $\tilde a\in [0,4]$ then $f_{\tilde a}$ has an absolutely continuous $f_{\tilde a}$-invariant probability measure with positive entropy (see \cite{ColletEckmann, Nowicki1985} or \cite[Chapter V-4]{de2012one}). Hence, by (1) of Proposition \ref{prop Hofbauer Keller} and Lemma \ref{lemma Nowicki and Sands}, we obtain $\llya(\map(c))=0$. Combining this with \eqref{eq upper lyapunov critical} and \eqref{eq delta limit}, we obtain (3) of Theorem \ref{thm quadratic map}.
\qed

\subsection*{Acknowledgments}
The author would like to thank Yushi Nakano for introducing this subject to the author and for numerous valuable discussions, and Johannes Jaerisch for numerous valuable discussions. The author would also like to thank Hiroki Takahasi for bringing Keller's paper \cite{Keller1990} to the author's attention.
The author was supported by the JSPS KAKENHI 25KJ1382.





\bibliographystyle{abbrv}
\bibliography{reference}

@article {Lyapunovholomorphic2016, 
AUTHOR = {Levin, Genadi and Przytycki, Feliks and Shen, Weixiao}, TITLE = {The {L}yapunov exponent of holomorphic maps}, JOURNAL = {Invent. Math.}, FJOURNAL = {Inventiones Mathematicae}, VOLUME = {205}, YEAR = {2016}, NUMBER = {2}, PAGES = {363--382}, ISSN = {0020-9910,1432-1297}, MRCLASS = {37F10 (37F15 37F50)}, MRNUMBER = {3529117}, MRREVIEWER = {Alexandre\ De Zotti}, DOI = {10.1007/s00222-015-0637-1}, URL = {https://doi.org/10.1007/s00222-015-0637-1}, }

@article {ColletEckmannholomorphic1998,
AUTHOR = {Przytycki, Feliks}, TITLE = {Iterations of holomorphic {C}ollet-{E}ckmann maps: conformal and invariant measures. {A}ppendix: on non-renormalizable quadratic polynomials}, JOURNAL = {Trans. Amer. Math. Soc.}, FJOURNAL = {Transactions of the American Mathematical Society}, VOLUME = {350}, YEAR = {1998}, NUMBER = {2}, PAGES = {717--742}, ISSN = {0002-9947,1088-6850}, MRCLASS = {58F23 (28D99 30D05 58F11)}, MRNUMBER = {1407501}, MRREVIEWER = {Walter\ Bergweiler}, DOI = {10.1090/S0002-9947-98-01890-X}, URL = {https://doi.org/10.1090/S0002-9947-98-01890-X}, }

@article {characteristic1993, 
AUTHOR = {Przytycki, Feliks}, TITLE = {Lyapunov characteristic exponents are nonnegative}, JOURNAL = {Proc. Amer. Math. Soc.}, FJOURNAL = {Proceedings of the American Mathematical Society}, VOLUME = {119}, YEAR = {1993}, NUMBER = {1}, PAGES = {309--317}, ISSN = {0002-9939,1088-6826}, MRCLASS = {58F23 (30D05)}, MRNUMBER = {1186141}, MRREVIEWER = {Fuyao\ Ren}, DOI = {10.2307/2159858}, URL = {https://doi.org/10.2307/2159858}, }

@article {NowickiSands1998, 
AUTHOR = {Nowicki, Tomasz and Sands, Duncan}, TITLE = {Non-uniform hyperbolicity and universal bounds for {$S$}-unimodal maps}, JOURNAL = {Invent. Math.}, FJOURNAL = {Inventiones Mathematicae}, VOLUME = {132}, YEAR = {1998}, NUMBER = {3}, PAGES = {633--680}, ISSN = {0020-9910,1432-1297}, MRCLASS = {58F15 (58F03)}, MRNUMBER = {1625708}, MRREVIEWER = {Grzegorz\ \'Swi\polhk atek}, DOI = {10.1007/s002220050236}, URL = {https://doi.org/10.1007/s002220050236}, }

@article {Keller1990, 
AUTHOR = {Keller, Gerhard}, TITLE = {Exponents, attractors and {H}opf decompositions for interval maps}, JOURNAL = {Ergodic Theory Dynam. Systems}, FJOURNAL = {Ergodic Theory and Dynamical Systems}, VOLUME = {10}, YEAR = {1990}, NUMBER = {4}, PAGES = {717--744}, ISSN = {0143-3857,1469-4417}, MRCLASS = {58F11 (58F12)}, MRNUMBER = {1091423}, MRREVIEWER = {Sebastian\ van Strien}, DOI = {10.1017/S0143385700005861}, URL = {https://doi.org/10.1017/S0143385700005861}, }

@book {de2012one,
    AUTHOR = {de Melo, Welington and van Strien, Sebastian},
     TITLE = {One-dimensional dynamics},
    SERIES = {Ergebnisse der Mathematik und ihrer Grenzgebiete (3) [Results
              in Mathematics and Related Areas (3)]},
    VOLUME = {25},
 PUBLISHER = {Springer-Verlag, Berlin},
      YEAR = {1993},
     PAGES = {xiv+605},
      ISBN = {3-540-56412-8},
   MRCLASS = {58F03 (58-02 58Fxx)},
  MRNUMBER = {1239171},
MRREVIEWER = {Feliks\ Przytycki},
       DOI = {10.1007/978-3-642-78043-1},
       URL = {https://doi.org/10.1007/978-3-642-78043-1},
}

@article {BruinRiveraShenStrien,
    AUTHOR = {Bruin, H. and Rivera-Letelier, J. and Shen, W. and van Strien,
              S.},
     TITLE = {Large derivatives, backward contraction and invariant
              densities for interval maps},
   JOURNAL = {Invent. Math.},
  FJOURNAL = {Inventiones Mathematicae},
    VOLUME = {172},
      YEAR = {2008},
    NUMBER = {3},
     PAGES = {509--533},
      ISSN = {0020-9910,1432-1297},
   MRCLASS = {37E05 (37C40 37D25)},
  MRNUMBER = {2393079},
MRREVIEWER = {Mike\ Todd},
       DOI = {10.1007/s00222-007-0108-4},
       URL = {https://doi.org/10.1007/s00222-007-0108-4},
}

@article {HofbauerKeller,
    AUTHOR = {Hofbauer, Franz and Keller, Gerhard},
     TITLE = {Quadratic maps without asymptotic measure},
   JOURNAL = {Comm. Math. Phys.},
  FJOURNAL = {Communications in Mathematical Physics},
    VOLUME = {127},
      YEAR = {1990},
    NUMBER = {2},
     PAGES = {319--337},
      ISSN = {0010-3616,1432-0916},
   MRCLASS = {58F11},
  MRNUMBER = {1037108},
MRREVIEWER = {Llu\'is\ Alsed\`a},
       URL = {http://projecteuclid.org/euclid.cmp/1104180141},
}

@article {Gukenheimer,
    AUTHOR = {Guckenheimer, John},
     TITLE = {Sensitive dependence to initial conditions for one-dimensional
              maps},
   JOURNAL = {Comm. Math. Phys.},
  FJOURNAL = {Communications in Mathematical Physics},
    VOLUME = {70},
      YEAR = {1979},
    NUMBER = {2},
     PAGES = {133--160},
      ISSN = {0010-3616,1432-0916},
   MRCLASS = {58F13 (34C35)},
  MRNUMBER = {553966},
MRREVIEWER = {Yu.\ E.\ Gliklikh},
       URL = {http://projecteuclid.org/euclid.cmp/1103907294},
}

@article {Singer,
    AUTHOR = {Singer, David},
     TITLE = {Stable orbits and bifurcation of maps of the interval},
   JOURNAL = {SIAM J. Appl. Math.},
  FJOURNAL = {SIAM Journal on Applied Mathematics},
    VOLUME = {35},
      YEAR = {1978},
    NUMBER = {2},
     PAGES = {260--267},
      ISSN = {0036-1399},
   MRCLASS = {58F20 (58C25)},
  MRNUMBER = {494306},
MRREVIEWER = {R.\ H.\ Atkin},
       DOI = {10.1137/0135020},
       URL = {https://doi.org/10.1137/0135020},
}

@article {Nowicki1985,
    AUTHOR = {Nowicki, Tomasz},
     TITLE = {Symmetric {$S$}-unimodal mappings and positive {L}iapunov
              exponents},
   JOURNAL = {Ergodic Theory Dynam. Systems},
  FJOURNAL = {Ergodic Theory and Dynamical Systems},
    VOLUME = {5},
      YEAR = {1985},
    NUMBER = {4},
     PAGES = {611--616},
      ISSN = {0143-3857,1469-4417},
   MRCLASS = {58F11 (58F10)},
  MRNUMBER = {829861},
MRREVIEWER = {Massimo\ Campanino},
       DOI = {10.1017/S0143385700003199},
       URL = {https://doi.org/10.1017/S0143385700003199},
}

@article {ColletEckmann,
    AUTHOR = {Collet, P. and Eckmann, J.-P.},
     TITLE = {Positive {L}iapunov exponents and absolute continuity for maps
              of the interval},
   JOURNAL = {Ergodic Theory Dynam. Systems},
  FJOURNAL = {Ergodic Theory and Dynamical Systems},
    VOLUME = {3},
      YEAR = {1983},
    NUMBER = {1},
     PAGES = {13--46},
      ISSN = {0143-3857,1469-4417},
   MRCLASS = {58F11 (28D99)},
  MRNUMBER = {743027},
MRREVIEWER = {Peter\ Walters},
       DOI = {10.1017/S0143385700001802},
       URL = {https://doi.org/10.1017/S0143385700001802},
}

@book {DenkerGrillendergerSigmund,
    AUTHOR = {Denker, Manfred and Grillenberger, Christian and Sigmund,
              Karl},
     TITLE = {Ergodic theory on compact spaces},
    SERIES = {Lecture Notes in Mathematics},
    VOLUME = {527},
 PUBLISHER = {Springer-Verlag, Berlin-New York},
      YEAR = {1976},
     PAGES = {iv+360},
   MRCLASS = {28A65 (54H20)},
  MRNUMBER = {457675},
MRREVIEWER = {W.\ Parry},
}

@article {KirikiLiNakanosoma,
    AUTHOR = {Kiriki, Shin and Li, Xiaolong and Nakano, Yushi and Soma,
              Teruhiko},
     TITLE = {Abundance of observable {L}yapunov irregular sets},
   JOURNAL = {Comm. Math. Phys.},
  FJOURNAL = {Communications in Mathematical Physics},
    VOLUME = {391},
      YEAR = {2022},
    NUMBER = {3},
     PAGES = {1241--1269},
      ISSN = {0010-3616,1432-0916},
   MRCLASS = {37D25},
  MRNUMBER = {4405572},
MRREVIEWER = {Luciana\ Silva\ Salgado},
       DOI = {10.1007/s00220-022-04337-6},
       URL = {https://doi.org/10.1007/s00220-022-04337-6},
}

@article {OttYorke,
    AUTHOR = {Ott, William and Yorke, James A.},
     TITLE = {When {L}yapunov exponents fail to exist},
   JOURNAL = {Phys. Rev. E (3)},
  FJOURNAL = {Physical Review E. Statistical, Nonlinear, and Soft Matter
              Physics},
    VOLUME = {78},
      YEAR = {2008},
    NUMBER = {5},
     PAGES = {056203, 6},
      ISSN = {1539-3755,1550-2376},
   MRCLASS = {37C30 (37C70)},
  MRNUMBER = {2551379},
MRREVIEWER = {Paulo\ Varandas},
       DOI = {10.1103/PhysRevE.78.056203},
       URL = {https://doi.org/10.1103/PhysRevE.78.056203},
}

@article {YushiSomaYamamoto,
    AUTHOR = {Nakano, Yushi and Soma, Teruhiko and Yamamoto, Kodai},
     TITLE = {Observable {L}yapunov irregular sets for planar piecewise
              expanding maps},
   JOURNAL = {Discrete Contin. Dyn. Syst.},
  FJOURNAL = {Discrete and Continuous Dynamical Systems},
    VOLUME = {43},
      YEAR = {2023},
    NUMBER = {7},
     PAGES = {2737--2755},
      ISSN = {1078-0947,1553-5231},
   MRCLASS = {37C40 (37C05 37C30)},
  MRNUMBER = {4578564},
MRREVIEWER = {Peng\ Sun},
       DOI = {10.3934/dcds.2023027},
       URL = {https://doi.org/10.3934/dcds.2023027},
}

@article {Kellerliting,
    AUTHOR = {Keller, Gerhard},
     TITLE = {Lifting measures to {M}arkov extensions},
   JOURNAL = {Monatsh. Math.},
  FJOURNAL = {Monatshefte f\"ur Mathematik},
    VOLUME = {108},
      YEAR = {1989},
    NUMBER = {2-3},
     PAGES = {183--200},
      ISSN = {0026-9255,1436-5081},
   MRCLASS = {28D05 (54H20)},
  MRNUMBER = {1026617},
MRREVIEWER = {Vladimir\ L.\ Levin},
       DOI = {10.1007/BF01308670},
       URL = {https://doi.org/10.1007/BF01308670},
}

@article {FariaGuarino,
    AUTHOR = {de Faria, Edson and Guarino, Pablo},
     TITLE = {Real bounds and {L}yapunov exponents},
   JOURNAL = {Discrete Contin. Dyn. Syst.},
  FJOURNAL = {Discrete and Continuous Dynamical Systems},
    VOLUME = {36},
      YEAR = {2016},
    NUMBER = {4},
     PAGES = {1957--1982},
      ISSN = {1078-0947,1553-5231},
   MRCLASS = {37D25 (37E05 37E20 37F10)},
  MRNUMBER = {3411549},
MRREVIEWER = {Liviana\ Palmisano},
       DOI = {10.3934/dcds.2016.36.1957},
       URL = {https://doi.org/10.3934/dcds.2016.36.1957},
}

@article{kiriki2026non,
  title={Non-existence of {L}yapunov exponents in the {N}ewhouse domain},
  author={Kiriki, Shin and Li, Xiaolong and Nakano, Yushi and Soma, Teruhiko},
  journal={arXiv preprint arXiv:2604.10913},
  year={2026}
}

@article {Rivera-Letelier,
    AUTHOR = {Rivera-Letelier, Juan},
     TITLE = {Asymptotic expansion of smooth interval maps},
      NOTE = {Some aspects of the theory of dynamical systems: a tribute to
              Jean-Christophe Yoccoz. Vol. II},
   JOURNAL = {Ast\'erisque},
  FJOURNAL = {Ast\'erisque},
    NUMBER = {416},
      YEAR = {2020},
     PAGES = {33--63},
      ISSN = {0303-1179,2492-5926},
      ISBN = {978-2-85629-917-3},
   MRCLASS = {37E05 (37C40 37D25)},
  MRNUMBER = {4142456},
MRREVIEWER = {Peng\ Sun},
       DOI = {10.24033/ast},
       URL = {https://doi.org/10.24033/ast},
}

@article {Alejo,
    AUTHOR = {Barrio Blaya, Alejo and Jim\'enez L\'opez, V\'ictor},
     TITLE = {An almost everywhere version of {S}m\'ital's order-chaos
              dichotomy for interval maps},
   JOURNAL = {J. Aust. Math. Soc.},
  FJOURNAL = {Journal of the Australian Mathematical Society},
    VOLUME = {85},
      YEAR = {2008},
    NUMBER = {1},
     PAGES = {29--50},
      ISSN = {1446-7887,1446-8107},
   MRCLASS = {37E05 (37B99 37C05)},
  MRNUMBER = {2460862},
MRREVIEWER = {Sergiy\ F.\ Kolyada},
       DOI = {10.1017/S1446788708000645},
       URL = {https://doi.org/10.1017/S1446788708000645},
}
 \nocite{*}

\end{document}